\documentclass[11pt,reqno]{amsart}
\usepackage{amsfonts}
\usepackage{fancyhdr}
\usepackage{times}
\usepackage[T1]{fontenc}
\usepackage{mathrsfs}
\usepackage{latexsym}
\usepackage[dvips]{graphics}
\usepackage{epsfig}
\usepackage{hyperref, amsmath, amsthm, amsfonts, amscd, flafter,epsf}
\usepackage{amsmath,amsfonts,amsthm,amssymb,amscd}

\usepackage{longtable}
\usepackage{tabu}
\usepackage{array}
\usepackage{multirow}

\input amssym.def
\input amssym.tex
\usepackage{color}
\usepackage[parfill]{parskip}
\usepackage{mathrsfs} % use command \mathscr
\usepackage{geometry}                % See geometry.pdf to learn the layout options. There are lots.
\usepackage{epstopdf}

\usepackage{verbatim}

\newtheorem{thm}{Theorem}[section]
\newtheorem{cor}[thm]{Corollary}

\newtheorem{lem}[thm]{Lemma}

\newtheorem*{thm*}{Theorem}
\newtheorem*{con*}{Conjecture}
\newtheorem*{lem*}{Lemma}
\newtheorem{prop}[thm]{Proposition}

{ \theoremstyle{remark}\newtheorem*{remark}{Remark} } %   I added this, this wasn't in the original template of JNT. 

\newcommand{\bigslant}[2]{{\raisebox{.2em}{$#1$}\left/\raisebox{-.2em}{$#2$}\right.}}

\begin{document}
\baselineskip=17pt
\hbox{}
\medskip

\title{Waring's Problem in Finite Rings}
\author{Ye\c{s}\.im Dem\.iro\u{g}lu Karabulut}

\email{yesim.demiroglu@rochester.edu}

\address{Department of Mathematics, University of Rochester, Rochester, NY}

\thanks{This work was partially supported by NSA grant H98230-15-1-0319.}

\maketitle

\begin{abstract}
In this paper we obtain sharp results for Waring's problem over general finite rings, by using a combination of Artin-Wedderburn theory and Hensel's lemma and building on new proofs of analogous results over finite fields that are achieved using spectral graph theory. We also prove an analogue of S{\'a}rk{\"o}zy's theorem for finite fields.

\


\

\end{abstract}

\section{Introduction} 

Waring's problem has a very long history with an extensive literature spanning many different areas of mathematics. It seems likely to continue to be a central problem in mathematical research for some time to come. The original version of the problem was due to Edward Waring, and first appeared in \emph{Meditationes Algebraicae}, published in $1770$, where he states "Every number is the sum of $4$ squares; every number is the sum of $9$ cubes; every number is the sum of $19$ biquadrates [$4^{\text{\tiny th}}$ powers]; and so on.". Although we can not be sure exactly what Waring had in mind, we assume that he meant that for every $k \geqslant 2$, there exists a smallest positive integer $g(k)$, with the property that every $x \in \mathbb{N}$ can be written as a sum of at most $g(k)$ many $k^{\text{\tiny th}}$ powers of natural numbers, satisfying $g(2) = 4$, $g(3) = 9$, $g(4) = 19$, etc. The fact that $g(2)=4$ is equivalent to the four-squares theorem that was proven by Lagrange the same year Waring published his general conjecture, but was apparently already known to Diophantus who alluded to it nearly fifteen hundred years earlier in examples apppearing in his treatise \emph{Arithmetica} (see \cite{CS4}).

Hilbert proved the existence of $g(k)$ for arbitrary $k$ by evaluating a $25$-fold integral in $1909$, and as a result of the efforts of many mathematicians over a period of many years, explicit formulas for $g(k)$ have now been found for all but finitely many values of $k$ (see \cite{Wooley}), with multiple areas of mathematics such as the Gaussian integers and additive number theory developed in the process of obtaining these results. It then became natural to ask whether the bound $g(k)$ could be improved if instead of considering all $x \in \mathbb{N}$, one limited oneself to considering only those $x$ that are sufficiently large. We let $G(k)$ denote the smallest such improved bound, and the question of finding explicit values of $G(k)$ for arbitrary $k$ is still wide open, with the problem having been solved only for $k=2$ and $k=4$. One of the approaches for attacking this question involves finding upper bounds for $G(k)$ using variants of the Hardy-Littlewood circle method, which is a powerful analytic tool also motivated by Waring's problem (for an overview of the complete history of the problem, see \cite{Wooley}).

After considering this question in the context of $\mathbb{N}$, it seems natural to generalize it to working over $\mathbb{Z}_n$, traditionally referred to as \emph{Waring's problem} $\bmod \ n$. One can fix a $k$ and look for a result which gives the smallest $m=g(k, n)$ such that every number in $\mathbb{Z}_n$ can be written as a sum of at most $m$ many $k^{\text{\tiny th}}$ powers. This problem was first completely solved for $k = 2, 3$ by Small in \cite{CS2}, after which he observed that the same techniques yield a solution for all $k$, \cite{CS1}. \footnote{There are a few minor errors in the stated values of $g(k, n)$ that appear in \cite{CS1}, \cite{CS3}, \cite{CS2}. The correct values should be $g(4,41) = 2$, $g(5, 71) = 2$ and $g(5,101) = 2$.}

The main contribution of the current paper is to extend results for Waring's problem to the context of general finite rings. As far as we are aware, there do not exist other such results in the literature, so that our results appear to be the first known ones of this type. We get these generalizations by appealing to Artin-Wedderburn theory and Hensel's lemma, allowing us to build on results over finite fields obtained at an earlier point (see Section~\ref{fields section}) in our paper.

In particular, the following represents a sampling of the type of results found in Section~\ref{WP in Finite Rings}, (see Theorem~\ref{importantTable} for a more complete statement).

Let $R$ be a (not necessarily commutative) finite ring. Then any element of $R$ can be written as a sum of $n$ many $k^{\text{\tiny th}}$ powers in $R$
\begin{itemize}
\item where $n = 2$, if $k = 3$ as long as $3,4,7 \nmid \vert R \vert $ \vspace{-0.3cm}
\item where $n = 5$, if $k = 4$ as long as $2, 9 \nmid \vert R \vert $ \vspace{-0.3cm}
\item where $n = 3$, if $k = 5$ as long as $5, 11, 16 \nmid \vert R \vert $ \vspace{-0.3cm}
\item where $n = 7$, if $k = 6$ as long as $2, 3, 25 \nmid \vert R \vert $ \vspace{-0.3cm}  
\item where $n = 4$, if $k = 7$ as long as $7, 8 \nmid \vert R \vert $ \vspace{-0.3cm}
\end{itemize}
where one can lower the values of $n$ for any given $k$ as long as one is willing to exclude a few more numbers from dividing $\vert R \vert $.

In order to save space, we limit our explicit statements in Theorem~\ref{importantTable} to $3 \leqslant  k \leqslant 11$, but similar results could easily be obtained for any value of $k$ for which one has results over finite fields. This reduces the question of Waring's problem over general finite rings to the analogous questions over finite fields.

Note that Small's solution to Waring's problem $\bmod \ n$ when $n$ is a prime answers Waring's problem for a special class of finite fields. In addition, Small obtains a result for more general finite fields by applying an inequality from the theory of diagonal equations over such fields to prove that if $q > (k-1)^4$, then every element of  $\Bbb F_q$ can be written as a sum of two $k^{\text{\tiny th}}$ powers in \cite{CS3} (which when translated into the language of this paper says $\gamma (k, q) \leqslant 2$ whenever $q>(k-1)^4$). In Section~\ref{fields section}, we deduce an asymptotically equivalent result to his and to results of \cite{Weil1949}, \cite{Winterhof1} by applying graph theoretic methods yielding self contained elementary proofs that when $q>k^4$, every element of the field can be written as a sum of two $k^{\text{\tiny th}}$ powers, and when $q > k^3$, as three $k^{\text{\tiny th}}$ powers, etc. We also enlist the help of a supercomputer to calculate the values for $\gamma (k, q)$ for all $q$ too small to be covered by the above mentioned results when $k \leqslant 37$.

On top of this, in the process of proving our finite field Waring's problem results (similar to results that were already known), we also obtain the following new result providing an analogue of S{\'a}rk{\"o}zy's theorem in the finite field setting.

\begin{thm*} 
Let $k$ be a positive integer. If $E$ is a subset of $\Bbb F_q$ with size $|E| > \dfrac{qk}{\sqrt{q-1}}$ then it contains at least two distinct elements whose difference is a $k^{\text{\tiny th}}$ power. Thus, in particular if $k$ is fixed and $|E|=\Omega \left(\sqrt{q} \right)$ as $q \rightarrow \infty$, then $E$ contains at least two distinct elements whose difference is a $k^{\text{\tiny th}}$ power.
\end{thm*}

Our main tool in the proofs of our finite field and S{\'a}rk{\"o}zy's results is a spectral gap theorem for Cayley digraphs for which we provide a self-contained proof in Section~\ref{spectral section}. By building on these finite field results, we also obtain in Section~\ref{matrix section} upper bounds for Waring numbers for matrix rings over finite fields (hence over general semisimple finite rings) that are stronger than our aforementioned results obtained in the case of more general finite rings. Results relating to matrix rings exist in the literature (see the history section and references in \cite{Katre} for a summary), but they often focus on existence results, showing that a given matrix over a commutative ring can be written as a sum of some number of $k^{\text{\tiny th}}$ powers, rather than determining how many $k^{\text{\tiny th}}$ powers need to be used. As was the case for our more general results mentioned above, our results over matrix and semisimple finite rings also rely on an efficient use of Hensel's Lemma, (see Theorem~\ref{matrixTable} for more details).

\section{Preliminaries}

\subsection{Spectral Graph Theory} \label{spectral section}

The material presented here prior to Cayley digraphs is pretty standard and can be found in any introductory graph theory book, e.g. \cite{Godsil}, \cite{West}. 

A \emph{graph} $G$ consists of a \emph{vertex set} $V(G)$, an \emph{edge set} $E(G)$, and a relation that associates with each edge two vertices called its \emph{endpoints}. A \emph{directed graph} or \emph{digraph} is a graph in which the endpoints of each edge is an ordered pair of vertices. If $(u,v)$ also denoted by $u \rightarrow v$ is such an ordered pair of vertices, that means there exists an edge in the digraph with \emph{tail} $u$ and \emph{head} $v$. When we draw a digraph, we represent the vertices with some dots and the directed edges between the vertices with some curves on the plane, and we give each edge a direction from its tail to its head. We define the \emph{out-degree} of a vertex $u$, denoted by $d^{+}(u)$, as the number of edges with tail $u$ and the \emph{in-degree}, $d^{-}(u)$, as the number of edges with head $u$. Each loop (i.e. an edge have the same endpoints) contributes $+1$ to both in- and out-degree. As a side note, notice that the summation of in-degrees and out-degrees are always equal to each other, and equals to the number of edges. We define the \emph{adjacency matrix} $\mathbb{A}_{G}$ (or $\mathbb{A}$) of a digraph $G$ as a matrix with rows and columns indexed by the vertices of $G$, such that the $uv$-entry of $\mathbb{A}$ is equal to the number of edges with tail $u$ and head $v$. Since $\mathbb{A}_{G}$ is not necessarily symmetric, unfortunately we lose the spectral theorem of linear algebra in most cases when $G$ is a digraph, but we can still consider the spectrum of $\mathbb{A}_{G}$ and it turns out to be still very useful.

If $G$ is a digraph, a \emph{walk} of length $r$ in $G$ can be defined as a sequence of vertices $( v_0, v_1, \cdots, v_r)$ with the only restriction that the ordered pair $(v_{i-1},v_{i})$ is the endpoints of some edge for every $1 \leqslant i \leqslant r$. That is to say if we want to walk from $v_{i-1}$ to $v_{i}$, we need the existence of an edge with the tail $v_{i-1}$ and head $v_{i}$, i.e. the direction matters in digraphs in contrast to graphs. On the other hand the definition of walk allows us to go over any edge and any vertex more than once. By induction on $k$ we can show that the $uv$-entry of $\mathbb{A}^k$, where $\mathbb{A}$ denotes the adjacency matrix of $G$, counts the (directed) walks of length $k$ from $u$ to $v$. It is also easy to show that if $G$ is a simple digraph, i.e. a digraph with no multi-edges and with at most one loop, then the trace of $\mathbb{A} \mathbb{A}^{T}$ equals $\sum_{u \in G} d^{+}(u)=\vert E(G) \vert$. Here multi-edge means there exists at least two edges which correspond to the same endpoints $u \rightarrow v$. 

Let $H$ be a finite abelian group and $S$ be a subset of $H$. The \emph{Cayley digraph} $Cay(H,S)$ is the simple digraph whose vertex set is $H$ and $u \rightarrow v$ if and only if $v-u \in S$. By definition $Cay(H,S)$ is a simple digraph with $d^{+}(u)=d^{-}(u)=|S|$. Furthermore, if we have a Cayley digraph, then we can find its spectrum easily using characters from representation theory, see \cite{Serre} for a rigorous treatment of character theory. A function $\chi: H \longrightarrow \mathbb{C}$ is a \emph{character} of $H$ if $\chi$ is a group homomorphism from $H$ into the multiplicative group $\mathbb{C}^{\ast}$. If $\chi(h)=1$ for every $h\in H$, we say $\chi$ is the \emph{trivial character}. The following theorem is a very important well-known fact, see e.g. \cite{Brouwer-Spectra}: 
\begin{thm} \label{character theory}
Let $\mathbb{A}$ be an adjacency matrix of a Cayley digraph $Cay(H,S)$. Let $\chi$ be a character on $H$. Then the vector $(\chi(h))_{h \in H}$ is an eigenvector of $\mathbb{A}$, with eigenvalue $\sum_{s \in S} \chi(s)$. In particular, the trivial character corresponds to the trivial eigenvector $\boldsymbol{1}$ with eigenvalue $|S|$.
\end{thm}

\begin{proof} 
Let $u_1,u_2, \cdots, u_n$ be an ordering of the vertices of the digraph and let $\mathbb{A}$ correspond to this ordering. Pick any $u_i$. Then we have \[ \sum_{j=1}^{n} \mathbb{A}_{ij} \chi(u_j)= \sum_{u_{i} \rightarrow u_{j}} \chi(u_{j})=\sum_{s \in S} \chi(u_{i}+s)= \sum_{s \in S} \chi(u_{i})\chi(s)=\left[\sum_{s \in S} \chi(s)\right]\chi(u_{i}).\] Thus \[ \mathbb{A} \chi= \left[\sum_{s \in S} \chi(s)\right] \chi .\]
\end{proof}

Notice that we get $|H|$ many eigenvectors from the characters as demonstrated in the previous theorem, and they are all distinct since they are orthogonal by \emph{character orthogonality}, \cite{Serre}. This means adjacency matrix $\mathbb{A}$ of a Cayley digraph is diagonalizable and we know all of the eigenvectors explicitly, assuming we know all of the characters of $H$. This shows us that $\mathbb{A}$ is non-defective and the eigenvectors of $\mathbb{A}$ form an eigenbasis. Indeed, it will be very important in the proof of the spectral gap theorem (upcoming). Recall when we have a (undirected) graph, since we have $\mathbb{A}$ real and symmetric, we know there exists an eigenbasis of $\mathbb{A}$ by the spectral theorem of linear algebra. On the other hand, this is not the case for digraphs in general. For example, consider the digraph with vertex set $V(G)=\{u,v\}$ and with only one directed edge $u \rightarrow v$. In this case $G$ is a simple digraph with no loops; $\mathbb{A}_{G}=\begin{bmatrix} 0 & 1  \\ 0 & 0 \end{bmatrix}$ has only one eigenvector, so $\mathbb{A}_{G}$ does not have an eigenbasis. This makes Cayley digraphs nicer than general simple digraphs. 

Theorem~\ref{SGT} (viz. spectral gap theorem) below is a very important and widely used tool in graph theory by itself. Hence we intend to give a self-contained proof.

\begin{thm}[Spectral Gap Theorem For Cayley Digraphs] \label{SGT}
Let $Cay(H,S)$ be a Cayley digraph of order $n$. Let $\{ \chi_i \}_{i=1,2, \cdots, n}$ be the set of all distinct characters on $H$ such that $\chi_{1}$ is the trivial one. Define \[ n_{\ast}= \frac{n}{|S|} \left( \max_{2 \leqslant i \leqslant n} \Big\| \sum_{s \in S} \chi_{i}(s) \Big\| \right) \] and let $X,Y$ be subsets of vertices of $Cay(H,S)$. If $\sqrt{|X||Y|} > n_{\ast}$, then there exists a directed edge between a vertex in $X$ and a vertex in $Y$. In particular if $|X|> n_{\ast}$, then there exists at least two distinct vertices of $X$ with a directed edge between them.
\end{thm}

\begin{proof} 
Let $u_1,u_2, \cdots, u_n$ be an ordering of the vertices of $Cay(H,S)$ and let $\mathbb{A}$ be the adjacency matrix corresponding to this ordering. By Theorem~\ref{character theory} each $\chi_i$ gives us an eigenvector for $\mathbb{A}$ with an eigenvalue $\sum_{s \in S} \chi_{i}(s) $. We can normalize $\chi_i$ for each $i \in \{1,2, \cdots, n \}$ and get $v_i$'s so that $v_1, \cdots , v_n$ is an orthonormal basis for $\mathbb{C}^n$. Therefore, any complex $n$-dimensional vector $v$ can be written as $v= \sum^n_{j=1} \langle v,v_j\rangle v_j$ where $\langle-,-\rangle$ denotes the Hermitian inner product in $\mathbb{C}^n$. Notice 
\begin{equation} \label{Planch}
\langle v,v \rangle = \left\langle \sum^n_{j=1}   \langle v, v_j \rangle v_j , v \right \rangle =\sum^n_{j=1}  \langle v, v_j \rangle  \langle v_j , v \rangle =\sum_{j=1}^{n} \|\langle v,v_j\rangle\|^2 
\end{equation} which is also known as the Plancherel identity. 

Define $\boldsymbol{1_{X}}$ as the $0$-$1$ column vector whose $i^{\text{\tiny th}}$ entry is $1$ when $u_{i} \in X$ and $0$ otherwise. Define $\boldsymbol{1_{Y}}$ similarly. If we calculate
\[\mathbb{A}\boldsymbol{1_{Y}}=\begin{bmatrix} a_1\\ a_2\\ \vdots \\a_n \end{bmatrix}\] then each $a_i$ denotes the number of directed edges from $u_{i}$ to the vertices of $Y$. If we multiply both sides of the equation with $\boldsymbol{1^T_{X}}$ from the left, we get \[\boldsymbol{1^T_{X}} \mathbb{A} \boldsymbol{1_{Y}}= \boldsymbol{1^T_{X}} \begin{bmatrix} a_1\\ a_2\\ \vdots \\a_n \end{bmatrix}= \sum_{u_i \in X} a_{i}\] which is indeed exactly the number of edges from the vertices of $X$ to the vertices of $Y$. This calculation shows that as long as $\boldsymbol{1^T_{X}} \mathbb{A} \boldsymbol{1_{Y}} > 0$ there will be a directed edge from a vertex in $X$ to a vertex in $Y$.

On the other hand we also have  
\begin{align} 
\boldsymbol{1^T_{X}} \mathbb{A} \boldsymbol{1_{Y}}&= \langle \boldsymbol{1_{X}},\mathbb{A} \boldsymbol{1_{Y}} \rangle = \left\langle \boldsymbol{1_{X}}, \mathbb{A} \left( \sum_{j=1}^n \langle \boldsymbol{1_{Y}}, v_{j} \rangle v_{j} \right) \right\rangle = \left\langle \boldsymbol{1_{X}}, \left( \sum_{j=1}^n \langle \boldsymbol{1_{Y}}, v_j \rangle \lambda_{j} v_{j} \right) \right\rangle \nonumber \\ &=\sum_{j=1}^n \left\langle \boldsymbol{1_{X}}, \langle \boldsymbol{1_{Y}}, v_{j} \rangle \lambda_{j} v_{j} \right\rangle= \sum_{j=1}^n \langle \boldsymbol{1_{X}},\lambda_{j} v_j \rangle \overline {\langle \boldsymbol{1_{Y}}, v_{j} \rangle} =\sum_{j=1}^n  \overline {\lambda_{j}}  \langle \boldsymbol{1_{X}}, v_{j} \rangle \langle v_{j},  \boldsymbol{1_{Y}} \rangle\label{eqn1} 
\end{align}

By Theorem~\ref{character theory} we know one of the eigenvalues of $Cay(H,S)$ is $\lambda_1=|S|$ with eigenvector $v_1= \frac{1}{\sqrt{n}}\boldsymbol{1}$ and by substitution in \eqref{eqn1} we get
\[\boldsymbol{1^T_{X}} \mathbb{A} \boldsymbol{1_{Y}} = \frac{|X||Y||S|}{n} + \sum_{j=2}^n \overline {\lambda_{j}}  \langle \boldsymbol{1_{X}}, v_j \rangle \langle v_{j} , \boldsymbol{1_{Y}} \rangle.\] Let \[E = \sum_{j=2}^n \overline{\lambda_j} \langle  v_j , \boldsymbol{1_{Y}} \rangle \langle \boldsymbol{1_{X}}, v_j \rangle .\] We will show if $\sqrt{|X||Y|} > n_{\ast}$ then $\dfrac{|X||Y||S|}{n} > |E|$, and the result will follow.

By the Cauchy-Schwarz inequality we have \[ |E| \leqslant ( \max_{2 \leqslant j \leqslant n} \| \lambda_j \| ) \left(\sum_{j=1}^n \| \langle \boldsymbol{1_{Y}} ,v_j \rangle \|^2 \right)^{1/2} \left(\sum_{j=1}^n \| \langle \boldsymbol{1_{X}} ,v_j \rangle \|^2 \right)^{1/2} \] 

Moreover by Plancherel's equality we have \[ \lvert E \rvert \leqslant \left( \max_{2 \leqslant j \leqslant n} \| \lambda_j \| \right) \sqrt{\lvert X \rvert \lvert Y \rvert}.\] Hence, the result follows as long as \[ \frac{|S| \lvert X \rvert \lvert Y\rvert}{n} > \left( \max_{2 \leqslant j \leqslant n} \| \lambda_j \| \right) \sqrt{\lvert X \rvert \lvert Y \rvert} \] holds. The only thing left to point out is that in the statement of the theorem we substituted $\lambda_j$'s using Theorem~\ref{character theory}.
\end{proof}
Keep the notation in the previous theorem for the next corollary. Define \[ n_{\ast,k}= \frac{n}{|S|^{k}} \left( \max_{2 \leqslant i \leqslant n} \Big\| \sum_{s \in S} \chi_{i}(s) \Big\| \right)^{k}.\]
\begin{cor} \label{fancy spectral gap theorem}
If $n_{\ast,k} < \sqrt{|X||Y|}$, then there exists a directed $k$-walk from a vertex in $X$ to a vertex in $Y$. In particular if $n_{\ast,k} < |X|$, then there exists at least two distinct vertices of $X$ with a directed $k$-walk between them.
\end{cor}

The proof of the corollary follows from a well-known trick in graph theory. We form a new graph with the same vertex set $Cay(H,S)$, where every directed edge from $u$ to $v$ in the new graph corresponds to an oriented walk of length $k$ from $u$ to $v$ in the original digraph. This new graph can be most easily defined as the digraph with the same vertex set as new graph and adjacency matrix $\mathbb{A}^k$ where $\mathbb{A}$ is adjacency matrix of the original graph. Notice $\mathbb{A}^k$ has the same orthogonal eigenbasis $\mathbb{A}$ does and its eigenvalues are the $k^{\text{\tiny th}}$ powers of $\mathbb{A}$'s. So, the corollary follows when we apply the theorem on the new graph. Furthermore, notice that we can prove this theorem for any digraph with orthogonal eigenbasis, and not just for Cayley digraphs.
  
\begin{cor} [Spectral Gap Theorem For $d$-regular Graphs]
Let $G$ be a $d$-regular simple graph with eigenvalues $\lambda_{1} \geqslant \cdots \geqslant \lambda_{n}$. Let \[ n_{\ast}= \frac{n}{d} \left( \max_{2 \leqslant i \leqslant n} | \lambda_i | \right) \] and let $X,Y \subset V(G)$ be subsets of vertices. If $\sqrt{|X||Y|} > n_{\ast}$, then there exists an edge incident to a vertex in $X$ and a vertex in $Y$. In particular if $|X|> n_{\ast}$, then there exists at least two vertices of $X$ adjacent to each other.
\end{cor}

This corollary does not follow from the statement of the theorem directly, but it follows from the proof. There are only a few points to change in the theorem's proof to prove the corollory. First, we do not always know the eigenvalues and eigenvectors of a regular simple graph explicitly. But we do know if $G$ is a $d$-regular simple graph with $n$ vertices, then it should have $\lambda=d$ as an eigenvalue with multiplicity at least $1$ and there exists an eigenbasis for $\mathbb{R}^n$. Second, since the $\boldsymbol{1^T_{X}} \mathbb{A} \boldsymbol{1_{Y}}$ term in the theorem's proof was the number of directed edges; it will be the number of edges between the vertices of $X$ and vertices of $Y$ such that the edges in the intersection counted twice, where \emph{the edges in the intersection} refers to the edges which has the endpoints in $X \cap Y$. However, the rest of the proof works as demonstrated above.

\subsubsection{Cayley Graphs and Character Theory}

We need to use characters only for finite fields in this paper. Let $\Bbb F_p$ be the finite field of prime order $p$. Every time we fix a $j \in \{ 0,1, \cdots, p-1 \}$ and define $\chi_{j}(x) := e^{\frac{2\pi ijx}{p}}$ for every $x \in \Bbb F_p$, $\chi_{j}$ becomes a character on the additive group structure of $\Bbb F_p$, \cite{Rosen}. Since different $j$ values provide us different characters on $\Bbb F_p$, we have $p$ many characters defined in this way, but the theory tells us we should have $p$ many characters for a group of order $p$, so we obtained all the characters on $\Bbb F_p$.

Let $\Bbb F_q$ be the finite field of order $q$ with characteristic $p$ and $\Bbb F_q^{\ast}$ denote the multiplicative group of it as usual. We want to give a characterization for the characters on $\Bbb F_q$. Once we know the characters for the base field $\Bbb F_p$, the characters of $\Bbb F_q$ can be obtained by following the process we are about to explain. Let $\alpha \in \Bbb F_q^{\ast}$ and define a map $\phi_{\alpha} : \Bbb F_q \longrightarrow \Bbb F_p$ as $\phi_{\alpha}(x) := \operatorname{tr}(\alpha x)$ for every $x \in \Bbb F_q$. Here $\alpha x$ denotes the multiplication of $\alpha$ and $x$ in $\Bbb F_q$, and $\operatorname{tr}$ is the usual field trace of $\Bbb F_q$ over $\Bbb F_p$. We can easily show that for any element $\alpha \in \Bbb F_q^{\ast}$ and for any character $\chi_{j}$ on $\Bbb F_p$, $\chi_{j} \circ \phi_{\alpha}$ defines a character on $\Bbb F_q$. We can also show that $\chi_{j} \circ \phi_{\alpha}$ is the trivial character on $\Bbb F_q$ if and only if $j$ is zero. Then also notice that $\chi_{j} \circ \phi_{\alpha}$ and $\chi_{j} \circ \phi_{\beta}$ is the same character if and only if $\alpha = \beta$. Hence, if we fix $j=1$ and if we range $\alpha$ in $\Bbb F_q^{\ast}$, $\chi_1 \circ \phi_{\alpha}$'s become distinct. That is why we have $q-1$ many distinct nontrivial characters by ranging $\alpha$. We also have the trivial character, and we can think of the trivial character as $j=1$ and $\alpha=0$. In this way we obtain at least $q$ many distinct characters on $\Bbb F_q$, but from the theory we know we should have exactly $q$ many characters on $\Bbb F_q$, which means \emph{any} character on $\Bbb F_q$ can be written explicitly as $\chi_{1} \circ \phi_{\alpha}$ for some $\alpha \in \Bbb F_q$.

\subsection{Finite Rings, Jacobson Radical and Artin-Wedderburn Theory}

As we stated before, the main purpose of this paper is to solve Waring's problem for finite rings with identity. We actually have a wide range of such rings. Some important examples are finite fields, the ring of integers modulo $n$ for each integer $n\geqslant 2$, matrix rings over finite fields, which we will denote by $\Bbb F_q$, $\mathbb{Z}_{n}$, $\operatorname{Mat}_n(\Bbb F_q)$, respectively. Some subrings of $\operatorname{Mat}_n(\Bbb F_q)$ (e.g. upper triangular matrices, lower triangular matrices, diagonal matrices etc., or $\Bbb F_q[A]$ for $A \in \operatorname{Mat}_n(\Bbb F_q)$) are also familiar examples. A particular class of examples, which will be used in Section~\ref{matrix section}, comes from quotients of polynomial rings over finite fields with some principal ideal generated by a polynomial over the same field $\left( \text{i.e. } \ \bigslant{\Bbb F_q[x]}{\langle f(x) \rangle} \right)$. If the polynomial is irreducible over $\Bbb F_q$, the quotient ring is isomorphic to a finite field. Otherwise it is not even an integral domain. In addition to those when we take a finite family of finite rings with identity, we can make the direct sum of underlying abelian groups into a ring by defining multiplication coordinate-wise and the resulting ring is a finite ring with identity.

Another important family of examples for finite rings with identity come from group rings. When $R$ is a ring with identity and $G$ is a group with operation $\ast$, the \emph{group ring} $R[G]$ is the set of all linear combinations of finitely many elements of $G$ with coefficients in $R$ i.e. any $x \in R[G]$ can be written as $\sum_{g \in G} a_{g} g$ where it is assumed that $a_g=0$ for all but finitely many elements of $G$. The addition is defined by the rule $\sum_{g \in G} a_{g} g + \sum_{g \in G} b_{g} g= \sum_{g \in G} (a_{g}+b_{g}) g$ and the multiplication is given by $\left(\sum_{g \in G} a_{g} g \right) \left( \sum_{g \in G} b_{g} g \right)=  \sum_{g \in G, \  h \in H} (a_{g}b_{h}) g \ast h$. The group ring of a finite group over any finite ring (or finite field) is a finite ring. When $R$ is a field, $R[G]$ is called a \emph{group algebra}. When the characteristic of the finite field divides the order of the group, the group algebra breaks up into finite rings called block rings.  

Artin-Wedderburn theory provides us very useful structure theorems for certain classes of rings. We now recall the basics of this theory. The omitted proofs are standard and can be found in most (noncommutative) algebra books presenting the subject. In particular see \cite{Farb-Algebra}, \cite{Hungerford} since our treatment follows theirs closely for the convention and notation.

Let $R$ be any ring, not necessarily finite. Then there is a left ideal $J(R)$, or just $J$, called the \emph{Jacobson radical} of $R$ such that $J=\bigcap\limits_{m \in \mathcal{M}}m$ where $\mathcal{M}$ is the set of all maximal left ideals. Let $x \in R$. It is a well-known fact that $x \in J$ if and only if $1+rxs$ is a unit in $R$ for all $r,s \in R$. We can also prove that $J$ is the intersection of all maximal right ideals of $R$. This combined with the definition of $J$  together implies that $J$ is a two-sided ideal. Also notice that $1 \notin J$ so $J \neq R$. 

A ring $R$ is said to be \emph{semisimple} if its Jacobson radical $J$ is zero. As we state later, Artin-Wedderburn theory gives us $\operatorname{Mat}_{n_1}(\Bbb F_{q_1}) \times \cdots \times \operatorname{Mat}_{n_j}(\Bbb F_{q_j})$, the direct product of finitely many matrix rings over finite fields, is a semisimple finite ring. Maschke's theorem shows many group algebras are semisimple, \cite{Curtis}:

\begin{thm}[Maschke's Theorem] \label{MT} 
Let $H$ be a finite group and let $F$ be a field whose characteristic does not divide the order of $H$. Then $F[H]$, the group algebra of $H$, is semisimple.
\end{thm}

In the case when the characteristic of the finite field divides the order of the group, Maschke's theorem does not apply and the group algebra is no longer semisimple. Instead it decomposes into block rings, which are another class of finite rings, with great importance in modular representation theory, see \cite{Alperin}.

We now state the main theorem in Artin-Wedderburn theory that shows in particular that any finite semisimple ring is a product of matrix rings.
\begin{thm}[Artin-Wedderburn Theorem] \label{AWT}
The following conditions on a ring $R$ are equivalent:
\vspace{-0.3cm}
\begin{itemize}
	\item $R$ is a nonzero semisimple left Artinian ring.
	\vspace{-0.3cm}
	\item $R$ is a direct product of a finite number of simple ideals each of which is isomorphic to the endomorphism ring of a finite dimensional vector space over a division ring.
	\vspace{-0.3cm}
	\item There exist division rings $D_1, \cdots, D_t$ and positive integers $n_1, \cdots, n_t$ such that $R$ is isomorphic to the ring $\operatorname{Mat}_{n_{1}}(D_1) \times \operatorname{Mat}_{n_{2}}(D_2) \times \cdots \times \operatorname{Mat}_{n_{t}}(D_t)$. 
\end{itemize}
\end{thm}
While in the general theory of Artin-Wedderburn division rings appear, we need not work with division rings as a result of the following theorem.
\begin{thm}[Wedderburn's Little Theorem] \label{WLT}
Every finite division ring $D$ is a field. 
\end{thm}
We will also use Nakayama's lemma: 
\begin{thm}[Nakayama's Lemma] \label{Nakayama}
If $I_{1}$ is an ideal in a ring $R$ with identity, then the following conditions are equivalent:
\vspace{-0.3cm}
\begin{itemize}
	\item $I_{1}$ is contained in the Jacobson radical, $J$
	\vspace{-0.3cm}
	\item If $I_{2}$ is a finitely generated $R$-module such that $I_{1}I_{2}=I_{2}$, then $I_{2}=0$.
\end{itemize}
\end{thm}

\begin{cor} \label{Nak2}
If $R$ is a finite ring, then $J^l=0$ for some $l \in \mathbb{Z}_{+}$. 
\end{cor}

\begin{proof}	
Since $R$ is finite, the following chain has to stop for some $l \in \mathbb{Z}_{+}$
\[J \supseteq J^2 \supseteq J^3 \supseteq \cdots \supseteq J^l =J^{l+1} = J^{l+2} = \cdots.\] 
If $J^l=0$, we are done. Otherwise, let $I_{1}=J$, $I_{2}= J^{l}$  and apply Nakayama's lemma. 
\end{proof}

Actually we do not need $R$ to be a finite ring to have $J^l=0$ for some $l$, if $R$ is just Artinian the result still follows by the same proof. We record this version since we need this result later only when $R$ is finite.

\subsection{Hensel's Lemma in Special Cases}
In this section we present and prove Hensel's lemma for the Jacobson radical of finite commutative rings with identity and polynomial rings over finite fields. These results of course can be deduced from the general theory easily, but we provide self contained proofs. Our approach follows the standard proof of Hensel's lemma in other settings, which usually resemble the classical \emph{Newton approximation method}, e.g. see Hensel's lemma for the $p$-adic integers in \cite{Alain}.

\subsubsection{Hensel's Lemma for the Jacobson Radical} 
Let $R$ be a ring with identity and let $J$ denote the Jacobson radical of $R$. First, we record two general basic facts in the following lemmata, since we need them later. 

\begin{lem} \label{Small Fact} 
If $a+J$ is a left unit in $\bigslant{R}{J}$, then $a$ is a left unit in $R$ and this implies $a+J^i$ is a left unit in $\bigslant{R}{J^i}$ for any $i \geqslant 1$.
\end{lem}

\begin{proof} 
If $a+J$ is a left unit in $\bigslant{R}{J}$, then there exists some $u \in R$ such that $(u+J)(a+J)=1+J$. That means there exists some $u \in R$ and $j \in J$ such that $ua=1+j$. This implies that there exists some $y \in R$ such that $yua=1$, which says $yu$ is a left inverse of $a$.\\
Also notice that $(yu+J^i)(a+J^i):=yua+J^i=1+J^i$ which means $a+J^i$ is a left unit in $\bigslant{R}{J^i}$.
\end{proof}

One can replace \emph{left} with \emph{right} or \emph{2-sided}, and both the statement and proof of the previous proposition extend \emph{mutatis mutandis}. As a side note also recall that the existence of a left inverse for an element of $R$ in general does not imply the existence of a right inverse, and vice versa. But finite rings are nicer in this regard, we can prove that if there exists a left (resp. right) inverse for $r \in R$, then the left (resp. right) inverse of $r$ is also the right (resp. left) inverse of $r$.

\begin{lem}[Taylor's expansion of a polynomial]\label{Taylor}
Let $R$ be a commutative ring and $p(x) \in R[x]$ be any polynomial. There exists a polynomial $p_2 \in R[x,y]$ such that $p(x+y)=p(x)+yp'(x)+y^2p_2(x,y).$
\end{lem}

\begin{proof}
Let $p(x)=\sum a_{n}x^{n}$. Then, \[ p(x+y)=\sum a_{n}(x+y)^{n}=\sum a_{n}(x^n+nx^{n-1}y+y^2(\cdots))=\sum a_{n}x^n+y\sum na_{n}x^{n-1}+y^2p_2(x,y).\] 
\end{proof}

\begin{prop} \label{Hensel-Jacobson} 
Let $R$ be a commutative finite ring with identity. Let $p(x) \in R[x]$ where $R[x]$ stands for the polynomial ring over $R$. Assume there exists a root of $p(x)$ in $\bigslant{R}{J}$, say $a_0 +J$ for some $a_0 \in R$, i.e. $p(a_0+J) \equiv 0 \bmod J$. If $p'(a_0)+J$ is a unit in $\bigslant{R}{J}$, then $r=a_0- \dfrac{p(a_0)}{p'(a_0)}$ satisfies
\vspace{-0.3cm}
\begin{enumerate}
\item $p(r)  \equiv 0 \pmod{J^2}$
\vspace{-0.3cm}
\item $r=a_0 \pmod{J}$
\vspace{-0.3cm}
\item $p'(a_0)$ is a unit in $R$ and $p'(r)$ is a unit in $R$. 
\end{enumerate}
\end{prop}

\begin{proof} 
Notice that the existence of $\dfrac{1}{p'(a_0)}$ follows from Lemma~\ref{Small Fact} and from the assumptions. Suppose we have a root of $p(x)$ in $\bigslant{R}{J^2}$, let's denote it with $r+J^2$. If we send $r+J^2$ using the canonical reduction map from $\bigslant{R}{J^2}$ to $\bigslant{R}{J}$, the image of $r+J^2$ should also be a root i.e. \[p(r+J^2) \equiv 0 \pmod{J^2} \ \ \implies \ \ p(r+J) \equiv 0 \pmod{J}. \]
If $r+J^2$ is a lift of $a_0+J$ then we should have $r+J=a_0+J$ in $\bigslant{R}{J}$ i.e. $r=a_0+j$ for some $j \in J$. By Lemma~\ref{Taylor}, we have 
\[ p(r)= p(a_0+j)=p(a_0) +p'(a_0) j + j^2 b \]
for some $b \in R$. We need $p(a_0+j) \equiv 0 \pmod{J^2}$ that means we need 
\[p(a_0)+p'(a_0).j \equiv 0 \pmod{J^2}. \]
So, if we let $j=-\dfrac{p(a_0)}{p'(a_0)}$ from the beginning, everything works out in these equations and the claims follow. 
\end{proof}
If $a_0+J$ is a root of $p(x)$ in $\bigslant{R}{J}$ and $p'(a_0)$ is a unit in $R$, then we can upgrade $a_0+J$ to a root $r+J^2$ in $\bigslant{R}{J^2}$ by the proposition. But notice that using the exact same process, we can lift $r+J^2$ in $\bigslant{R}{J^2}$ to some root $r_1+J^3$ in $\bigslant{R}{J^3}$ and so on. This process gives us a solution in each $\bigslant{R}{J^i}$ for any $i \geqslant 1$, and combined with Corollary~\ref{Nak2}, we obtain the following corollaries.

\begin{cor}
Let $R$ be a finite commutative ring with identity and $p(x) \in R[x]$. If there exists a root of $p(x)$, say $a_0+J$, in $\bigslant{R}{J}$ and $p'(a_0)+J$ is a unit in $\bigslant{R}{J}$, then there exists a root of $p(x)$ in $R$. 
\end{cor}

\begin{cor}\label{Hensel2}
Let $R$ be a finite commutative ring with identity. Let $\alpha \in R$. If $\alpha \equiv B_1^k + \cdots+ B_m^k$ $\bmod \ J$ in $\bigslant{R}{J}$ for some $B_1,B_2,\cdots,B_m \in R$ and $k B_1^{k-1}$ is a unit in $R$, then there exists some $\beta$ in $R$ such that $\alpha=\beta^k +B_2^k+\cdots+ B_m^k$ holds.
\end{cor}

\begin{proof}
Let $p(x)= x^k+B_2^k + \cdots+ B_m^k-\alpha$ and apply the previous corollary.
\end{proof}

\subsubsection{Hensel's Lemma for Polynomial Rings} 

Let $Q(t)$ be a polynomial in $\Bbb F_q[x][t]=\Bbb F_q[x, t]$. We can write $Q(t)$ as $a_0+a_1t+a_2t^2+ \cdots + a_nt^n$ for some $a_0, a_1,  \cdots a_n \in \Bbb F_q[x]$. This means we may evaluate $Q(t)$ at any $t \in \Bbb F_q[x]$ and get outputs in $\Bbb F_q[x]$. 

Let $f(x)$ be an irreducible polynomial in $\Bbb F_q[x]$. Let $g$ be any polynomial in $\Bbb F_q[x]$. We can define $f$-\emph{adic valuation of} $g$, say $\nu(g)$, such as
\[ \nu (g)=m \iff g \equiv 0 \pmod{f^m} \ \  \text{ and } \ \ g \not\equiv 0 \pmod{f^{m+1}} \ \ \text{ for some } m \in \mathbb{Z}_{+}.\] 
We also set $\nu(g)=0$ if $f \nmid g$. 

\begin{prop} \label{Hensel for poly-weak version} 
Let $g \in \Bbb F_q[x]$ satisfy $Q(g) \equiv 0 \pmod{f^{n}}$ for some $n \in \mathbb{Z}_{+}$. Suppose $m$ is a positive integer less than or equal to $n$. If $Q'(g) \not\equiv 0 \pmod{f}$, there exists a $g_{2} \in \Bbb F_q[x]$ such that $Q(g_{2}) \equiv 0 \pmod{f^{n+m}}$ and $g \equiv g_{2} \pmod{f^{n}}$. Moreover, $g_{2}$ is unique $\bmod{f^{n+m}}$.
\end{prop}

We will refer $g_{2}$ as a \emph{lifting} of $g$ in $\bmod{f^{n+m}}$.

\begin{remark}
We should note that the restriction $m \leqslant n$ in this proposition is not that important. Since if we want to find a solution modulo $f^{j}$ for some $j > n$, we can apply the proposition $j-n$ many times recursively for $m=1$ and lift a solution to modulo $f^{j}$. The same idea also applies to the next proposition. 
\end{remark}
\begin{proof} 
By setting $Q(g+f^{n}t) \equiv 0 \pmod{f^{n+m}}$, we can solve for $t$. By applying Lemma~\ref{Taylor} on $Q$ we know that 
\begin{align*} Q(g+f^{n}t) &= Q(g)+f^{n}t Q'(g)+(f^{n}t)^2 \tilde{Q}(g,f^nt) \ \ \ \ \ \text{ for some polynomial } \tilde{Q} \\
&\equiv Q(g)+f^{n}t Q'(g) \pmod{f^{n+m}} \ \ \ \ \ \text{since } n \geqslant m.
\end{align*}
By assumption $Q(g) \equiv 0 \pmod{f^n}$ i.e. $Q(g)= f^nh$ for some $h \in \Bbb F_q[x]$. So, we should have
\begin{align*}Q(g+f^{n}t) &\equiv f^nh+f^ntQ'(g) \pmod{f^{n+m}} \\
&\equiv f^n(h+tQ'(g))  \pmod{f^{n+m}}.
\end{align*}
Moreover, 
\begin{align*} Q(g+f^{n}t) \equiv 0 \pmod{f^{n+m}} &\iff f^n(h+tQ'(g))=f^{n+m}v \ \ \ \ \text{ for some } v \in \Bbb F_q[x]\\
&\iff f^n(h+tQ'(g)-f^{m}v)=0 \ \ \ \ \text{ for some } v \in \Bbb F_q[x].
\end{align*}
To have this, since $\Bbb F_q[x]$ is an integral domain, we need $h+tQ'(g)-f^{m}v$ to be zero, that means we need $h+tQ'(g)=f^{m}v$ which is the same with $t \equiv \dfrac{-h}{Q'(g)} \pmod{f^{m}}$. So, if we define $g_{2}=g-f^n \dfrac{h}{Q'(g)}=g- \dfrac{Q(g)}{Q'(g)}$ from the beginning, we are all set. 

The only point we did not justify here is why $Q'(g)^{-1}$ makes sense in $\bmod{f^m}$. Remember $Q'(g) \in \Bbb F_q[x]$. $\Bbb F_q[x]$ is a PID, that means we have B\'ezout's identity in $\Bbb F_q[x]$. By assumption we have $Q'(g) \not \equiv 0 \pmod{f}$, which means $f$ doesn't divide $Q'(g)$. Since $f$ is irreducible, the only divisors of $f$ are $1$ and itself. So, $\gcd(Q'(g),f)=1$ and this implies $\gcd(Q'(g),f^m)=1$. By B\'ezout, we have $AQ'(g)+Bf^m=1$ for some $A,B \in \Bbb F_q[x]$ which shows $Q'(g)^{-1}$ exits in $\bmod{f^m}$ for any $m \geqslant 1$.
\end{proof}

\begin{prop} \label{Hensel for poly-strong version} 
Let $g \in \Bbb F_q[x]$ be such that $Q(g) \equiv 0 \pmod{f^n}$ for some $n \in \mathbb{Z}_{+}$, and let $m$ denote $\nu (Q'(g))$. If $n > 2m$, then $g_{2}=g-\dfrac{Q(g)}{Q'(g)}$ satisfies
\vspace{-0.3cm}
\begin{enumerate}
\item \label{item1} $Q(g_{2}) \equiv 0 \pmod{f^{n+1}}$
\vspace{-0.3cm}
\item \label{item2} $g_{2} \equiv g \pmod{f^{n-m}}$
\vspace{-0.3cm}
\item \label{item3} $\nu (Q'(g_{2}))= \nu (Q'(g))=m$.
\end{enumerate}
\end{prop}

\begin{proof} 
Since $Q(g) \equiv 0 \pmod{f^n}$, there exists an $h_1 \in \Bbb F_q[x]$ such that $Q(g)=f^{n} h_1$. 
Similarly since $\nu (Q'(g))=m$, there exists an $h_2 \in \Bbb F_q[x] $ such that $Q'(g)=f^{m}h_2$ and $h_2 \not\equiv 0 \pmod{f}$, which implies $h_2$ is a unit in $\bmod{f^i}$ for any $i$.

To prove \eqref{item1} by applying Lemma~\ref{Taylor} on $Q$ we have 
\begin{align*} Q(g_2) &= Q\left(g-\frac{Q(g)}{Q'(g)}\right) \\
&= Q(g)+\left[\frac{-Q(g)}{Q'(g)}\right]Q'(g)+\left[\frac{Q(g)}{Q'(g)}\right]^2 \hat{Q}(\cdots) \ \ \ \ \ \text{ for some polynomial } \hat{Q} \\
&= (f^{n-m} h_1h_2^{-1})^2 \  \hat{Q}(\cdots) \\
&\equiv 0 \pmod{f^{n+1}}  \ \ \ \ \ \text{since } n > 2m.
\end{align*}

\eqref{item2} follows from 
\[g_{2} - g = \frac{-Q(g)}{Q'(g)} = \frac{-f^{n}h_1}{f^{m}h_2} = -f^{n-m}h_1h_2^{-1} \equiv 0 \pmod{f^{n-m}}.\]

To prove \eqref{item3} by applying Lemma~\ref{Taylor} on $Q'$ we have
\begin{align*} Q'(g_2) &= Q' \big( g+(g_2-g) \big) \\
&= Q'(g)+(g_2-g)Q''(g)+(g_2-g)^2 \check{Q}(\cdots) \ \ \ \ \ \text{ for some polynomial } \check{Q} \\
&= Q'(g) - \bigg[ \frac{Q(g)}{Q'(g)} \bigg] Q''(g) + \bigg[ \frac{Q(g)}{Q'(g)} \bigg]^2 \check{Q}(\cdots) \\
&= f^{m}h_2-f^{n-m}h_1h_2^{-1} Q''(g) + (f^{n-m}h_1h_2^{-1})^2  \check{Q}(\cdots) \\
&=f^m \big[ h_2- f^{n-2m}h_1h_2^{-1} Q''(g) + f^{2n-3m} (h_1h_2^{-1})^2 \check{Q}(\cdots)\big] \\ 
&\equiv 0 \pmod{f^{m}}  \ \ \ \ \ \text{since } n > 2m.
\end{align*}
Since $\big[h_2- f^{n-2m}h_1h_2^{-1} Q''(g) + f^{2n-3m} (h_1h_2^{-1})^2 \check{Q}(\cdots)\big] \not\equiv 0\pmod{f}$ we have $\nu \big( Q'(g_2) \big)=m$.
\end{proof}

\section{Waring's Problem in Finite Fields}\label{fields section}

Let $\Bbb F_q$ be the finite field with $q=p^s$ elements where $p$ is the characteristic of $\Bbb F_q$. We pick a $k \in \mathbb{Z}_{+}$ and denote the set of all $k^{\text{\tiny th}}$ powers in $\Bbb F_q$ by $R_k$ i.e. $R_k= \{ \alpha^k  \mid  \alpha \in \Bbb F_q \}$. Let $\gamma(k, q)$ denote the smallest $m$ such that every element of $\Bbb F_q$ can be written as a sum of $m$ many $k^{\text{\tiny th}}$ powers in $\Bbb F_q$. We are interested in finding $\gamma(k, q)$ for different $k$ and $q$ values, but unfortunately $\gamma(k, q)$ does not always exist. For a fixed value of $k$, if every element of $\Bbb F_q$ can be written as a sum of $k^{\text{\tiny th}}$ powers, we say $\Bbb F_q$ is \emph{coverable} by the $k^{\text{\tiny th}}$ powers. Otherwise, we call it an \emph{uncoverable field} for that specific $k$. 

Since $\Bbb F_2$ and $\Bbb F_3$ are coverable for any $k$, the smallest example of an uncoverable field is $\Bbb F_4$ with $k=3$. The only cubes in $\Bbb F_4$ are $\{0,1\}$ and any linear combination of $0$ and $1$ is either $0$ or $1$, this shows if $\alpha$ is a primitive cube root of unity in $\Bbb F_4$, then it cannot be written as a sum of cubes. Therefore, we say $\Bbb F_4$ is an uncoverable field for $k=3$ or cubes. This example is a very simple observation, but it actually illustrates the point very well which is if $R_k$ lies in a proper subfield of $\Bbb F_q$, then we cannot cover $\Bbb F_q$ with $k^{\text{\tiny th}}$ powers. Furthermore, the converse also holds meaning $\Bbb F_q$ is coverable if and only if the set of $k^{\text{\tiny th}}$ powers does not lie in a proper subfield, see \cite{Winterhof1}.

Before we proceed, notice that $\gamma(1, q)$ is trivially $1$. In the case of $k=2$, if $\Bbb F_q$ is a field with characteristic $2$, we have the Frobenius automorphism which implies every element of $\Bbb F_q$ is a square so that $\gamma(2, q)=1$. If $\operatorname{char}\Bbb F_q \neq 2$, then exactly half of the elements in $\Bbb F_q^{\ast}$ are squares and $\gamma(2, q)=2$ follows by the \emph{pigeonhole principle}. 

If we look at the $k=3$ case, then one can prove that except for $q=4,7$ in every finite field every element can be written as a sum of at most two cubes, see \cite{RoyBarb}. $\Bbb F_7$ requires three cubes since the only cubes in $\Bbb F_7$ are $\{0,-1,+1\}$, and thus $3$ is not the sum of any two cubes. $\Bbb F_4$ is uncoverable as we discussed above.

We can also easily single out a nice bound for $\gamma(k, q)$ when it exists, and also the extreme case i.e. when $\gamma(k, q)=1$:
\begin{prop} \label{gamma1 case}
Let $\Bbb F_q$ be the finite field with order $q$ and let $k$ be a positive integer. If $\Bbb F_q$ is coverable with the $k^{\text{\tiny th}}$ powers, then $\Bbb F_q \subseteq d R_k$ where $d=\gcd(k,q-1)$. In particular if $\gamma(k, q)$ exists, then it is less than or equal to $k$. Moreover, $\Bbb F_q=R_k$ if and only if $\gcd(k, q-1)=1$. 
\end{prop}

\begin{proof}
We have $R_k \subseteq 2R_k \subseteq \cdots \subseteq mR_k= (m+1)R_k=\cdots$ for some $m$ since $\Bbb F_q$ is finite. We assumed $\Bbb F_q$ is coverable this implies $mR_k=\Bbb F_q$. If $nR_k \neq (n+1)R_k$ for some $n$, that means there exists a nonzero element $\alpha$ in $(n+1)R_k$ but not in $nR_k$. Let $R_k^{\ast}:= R_k \setminus \{ 0 \} $. $R_k^{\ast}$ is a group under multiplication and it is a subgroup of $\Bbb F_q^{\ast}$. Notice that if we can write $\alpha$ as a sum of $n+1$ many $k^{\text{\tiny th}}$ powers, then any element in the coset of $\alpha$, i.e. in $\alpha R_k^{\ast}$, can be written as a sum of $n+1$ many $k^{\text{\tiny th}}$ powers. This implies if $\alpha \in (n+1)R_k^{\ast}$, then $\alpha R_k^{\ast} \subseteq (n+1)R_k^{\ast}$. This means the number of the strict inclusions in $R_k \subseteq 2R_k \subseteq \cdots \subseteq mR_k$ is at most one less than the number of cosets of $R_k^{\ast}$ inside $\Bbb F_q^{\ast}$. Moreover, since $\Bbb F_q^{\ast}$ is cyclic of order $q-1$ and $R_k^{\ast}$ is a subgroup of $\Bbb F_q^{\ast}$, $R_k^{\ast}$ has to be a cyclic group. Let $\Bbb F_q^{\ast}=\langle \gamma \rangle$, then $R_k^{\ast} = \langle \gamma^{k} \rangle$ and $ | R_k^{\ast} | = \frac{q-1}{\gcd(k,q-1)}$. So, the number of cosets is $\gcd(k,q-1)$. Also, it is clear that $\Bbb F_q^{\ast}=R_k^{\ast}$ if and only if $\gcd(k,q-1) =1$.
\end{proof}

This section proceeds like this. First, we prove that when a finite field is sufficiently large, every element of the field can be written as a sum of two $k^{\text{\tiny th}}$ powers, three $k^{\text{\tiny th}}$ powers etc. There exists some similar old results in the literature implying our result in \cite{CS3}, \cite{Weil1949}, \cite{Winterhof1}; but we will provide as elementary and self-contained proof as possible by following the graph theoretical approach. Our proof actually turns out to be very fruitful, since we also obtain an original result which is an analog of S\'ark\"ozy's theorem in a finite field setting. After S\'ark\"ozy's we provide the precise values for $\gamma(k,q)$ for $4 \leqslant k \leqslant 37$ for all finite fields (only except some of the $\gamma(k,q) \leqslant 3$). Those results were obtained with the help of a supercomputer.

\begin{thm} \label{Main Theorem}
If $q>k^4$, then every element of $\Bbb F_q$ can be written as a sum of two $k^{\text{\tiny th}}$ powers. 
\end{thm}

\begin{proof} 
We will use the spectral gap theorem to prove the result. Since $v=0$ is already a sum of two $k^{\text{\tiny th}}$ powers, let $v \in \Bbb F_q^{\ast}$. We have $d=\gcd(k,q-1)$ many cosets of $R_k^{\ast}$ inside $\Bbb F_q^{\ast}$ and $v \in vR_k^{\ast}$. The plan is to take $X=R_k^{\ast}$, $Y=vR_k^{\ast}$ and apply the spectral gap theorem. When we apply it, the theorem will guarantee the existence of a directed edge from $X$ to $Y$ under the assumption $q>k^4$. This means there exists some $u^k \in X$ and $vw^k \in Y$ for some $u,w,z \in \Bbb F_q^{\ast}$ such that $u^k+z^k=vw^k$ i.e. $v=(\frac{u}{w})^{k}+(\frac{z}{w})^{k}$, and the result follows. The only thing left is to figure out when the spectral gap theorem applies.
 
First, define a Cayley digraph $Cay(\Bbb F_q,R_k^{\ast})$ whose vertices are the elements of $\Bbb F_q$, and there exists a directed edge from $u$ to $v$ if and only if $v-u \in R_k^{\ast}$. Since $R_k^{\ast}$ does not contain $0$, $Cay(\Bbb F_q,R_k^{\ast})$ is loopless. We know every character on $\Bbb F_q$ corresponds to an eigenvector of the adjacency matrix, $\mathbb{A}$ of $Cay(\Bbb F_q,R_k^{\ast})$ by Theorem~\ref{character theory}. For instance if $\Bbb F_q^{\ast}=\langle \gamma \rangle$ we have $\Bbb F_q=\{0, \gamma, \gamma^2, \cdots, \gamma^{q-1}=1\}$, and any character $\chi$ gives us a unique eigenvector such that the first entry of the corresponding eigenvector is $\chi(0)$, the second entry of the eigenvector is $\chi(\gamma)$, the third entry is $\chi(\gamma^2)$ etc. We can even find the eigenvalues corresponding to those eigenvectors using the same theorem. Recall that we proved every character on $\Bbb F_q$ should be in the form of $\chi_{1} \circ \phi_{\alpha}$ for some $\alpha \in \Bbb F_q$ where $\phi_{\alpha}(x) = \operatorname{tr}(\alpha x)$ and $\chi_{1}(x)=e^{\frac{2\pi i x}{p}}$ for every $x \in \Bbb F_q$. Let $\lambda_{\alpha}$ denote the eigenvalue of $\mathbb{A}$ corresponding to the eigenvector induced from $\chi_{1} \circ \phi_{\alpha}$. By Theorem~\ref{character theory} we have 
\[ \lambda_{\alpha} = \sum_{x \in R_k^{\ast}} \chi_{1} \circ \phi_{\alpha}(x)= \sum_{x \in R_k^{\ast}} e^{\frac{2\pi i (\operatorname{tr}(\alpha x))}{p}}.\]
Notice if $\alpha R_k^{\ast}=\beta R_k^{\ast}$ we have \[ \sum_{x \in R_k^{\ast}} e^{\frac{2\pi i ( \operatorname{tr}(\alpha x))}{p}} = \sum_{x \in R_k^{\ast}} e^{\frac{2\pi i ( \operatorname{tr}(\beta x))}{p}}= \lambda_{\beta}.\]
This calculation shows us if $\alpha$ and $\beta$ are in the same coset of $R_k^{\ast}$ inside $\Bbb F_q^{\ast}$, the eigenvectors induced by $\chi_{1} \circ \phi_{\alpha}$  and $\chi_{1} \circ \phi_{\beta}$ correspond to the same eigenvalue. We also have the trivial character which induces $\boldsymbol{1}$ as an eigenvector with eigenvalue $|R_k^{\ast}|$, this means we have at most $\frac{|\Bbb F_q^{\ast}|}{|R_k^{\ast}|}+1$ many distinct eigenvalues. Thus the spectrum of $\mathbb{A}$ contains the largest eigenvalue $|R_k^{\ast}|$ with multiplicity $1$ and other eigenvalues, $ \lambda_{\alpha_{i}}$'s each with multiplicity $|R_k^{\ast}|$, since $|\alpha_{i} R_k^{\ast}|=|R_k^{\ast}|$. Therefore, we have 
\begin{equation} \label{lambda sum one}
\sum_{\alpha \in \Bbb F_q} \| \lambda_{\alpha} \|^2=|R_k^{\ast}|^2 + \sum_{i=1}^{d} \| \lambda_{\alpha_{i}} \|^2 |R_k^{\ast}|.
\end{equation}
At the same time we know 
\begin{equation} \label{lambda sum two}
\sum_{\alpha \in \Bbb F_q} \| \lambda_{\alpha} \|^2=\operatorname{tr}(\mathbb{A}^{\ast} \mathbb{A})
\end{equation} 
from basic linear algebra. Notice that we also have 
\begin{equation} \label{lambda sum three}
\operatorname{tr}(\mathbb{A}^{\ast} \mathbb{A})=\operatorname{tr}(\mathbb{A}^{T} \mathbb{A})=\operatorname{tr}(\mathbb{A} \mathbb{A}^{T})=\sum_{\alpha \in \Bbb F_q} d^{+}(\alpha)
\end{equation} as $\mathbb{A}$ is real $\mathbb{A}^{\ast}=\mathbb{A}^{T}$, and the last equality holds since $Cay(\Bbb F_q,R_k^{\ast})$ is a simple digraph. By Equations~\eqref{lambda sum one}, \eqref{lambda sum two} and \eqref{lambda sum three}  we have \[ |R_k^{\ast}|^2 + \sum_{i=1}^{d} \| \lambda_{\alpha_{i}} \|^2 |R_k^{\ast}|= q |R_k^{\ast}|.\]
Hence \[ \sum_{i=1}^{d} \| \lambda_{\alpha_{i}} \|^2  = q- |R_k^{\ast}|  \] which implies \[ \max_{1 \leqslant i \leqslant d} \| \lambda_{\alpha_{i}} \| \leqslant \sqrt{q- |R_k^{\ast} |} \ .\] 
It follows that \[ n_{\ast}= \frac{q}{|R_k^{\ast}|}\big( \max_{1 \leqslant i \leqslant d} \| \lambda_{\alpha_{i}} \| \big) \leqslant \frac{q \sqrt{q- |R_k^{\ast} |} }{ |R_k^{\ast}|}.\]

Now, let $x=d=\gcd(k,q-1)$ and $y=q$ in Lemma~\ref{induction1} in Appendix~\ref{Lemmata}. Notice that since $y$ is the order of the finite field, it is already in $\mathbb{Z}_{+} \setminus \{1\}$. Besides, there is no harm assuming $x \neq 1$, since we already handled $x=1$ case before in Proposition~\ref{gamma1 case}. By the lemma, we have $(y-1)^4-x^4y^3+(y-1)y^2x^3>0$, by dividing it with $x^4$ we get $\dfrac{(y-1)^4}{x^4}+\dfrac{(y-1)y^2}{x}>y^3$. This means $\lvert R_k^{\ast} \rvert^{4}+\lvert R_k^{\ast} \rvert q^2 >q^3$ which is to say $\lvert R_k^{\ast} \rvert^{4} >q^3-\lvert R_k^{\ast} \rvert q^2$ i.e. $\lvert R_k^{\ast} \rvert^{4} > q^2 (q-\lvert R_k^{\ast} \rvert)$. So $q>k^4$ implies $\lvert R_k^{\ast} \rvert > n_{\ast}$ and the result follows.
\end{proof}

\begin{thm} \label{gamma3}
If $q>k^3$, then every element of $\Bbb F_q$ can be written as a sum of three $k^{\text{\tiny th}}$ powers. 
\end{thm}

\begin{proof} 
We will use Corollary~\ref{fancy spectral gap theorem} and the discussion provided right after that. We define a new digraph using the Cayley digraph in the previous theorem's proof. The vertices of the new digraph are the elements of $\Bbb F_q$, and there exists a directed edge from $u$ to $v$ if and only if there is a $2$-walk from $u$ to $v$ in the original graph. Since $v=0$ is already a sum of three $k^{\text{\tiny th}}$ powers, let $v \in \Bbb F_q^{\ast}$. Let $X=R_k^{\ast}$, $Y=vR_k^{\ast}$ and apply Corollary~\ref{fancy spectral gap theorem}. It will guarantee the existence of a directed $2$-walk from $X$ to $Y$ under the assumption $q>k^3$ and the result will follow.

Let $x=d=\gcd(k,q-1)$ and $y=q$ in Lemma~\ref{induction2} in Appendix~\ref{Lemmata}. By the lemma, we have $(y-1)^3-x^{2}y(xy-y+1)>0$, by dividing it with $x^3$ we get $\frac{(y-1)^3}{x^3} > y \Big( \frac{xy-y+1}{x} \Big)$. This means $\lvert R_k^{\ast} \rvert^{3} > q (q-\lvert R_k^{\ast} \rvert)$ i.e. $\lvert R_k^{\ast} \rvert > \frac {q (q-\lvert R_k^{\ast} \rvert)}{\lvert R_k^{\ast} \rvert^{2}}$ and the result follows.
\end{proof}
\begin{remark} If one can prove that $y^{m-1}>x^{2m}$ implies $(y-1)^m-yx^{\frac{m+1}{2}}(xy-y+1)^{\frac{m-1}{2}}>0$ for every $m,x,y \in \mathbb{Z}_{+} \setminus \{1\}$, then using the same exact approach in the previous theorems' proof, we can state this \[ q> k^{\frac{2m}{m-1}}  \implies \Bbb F_q \subseteq mR_k \] for every integer $m>1.$ The required inequality $(y-1)^m-yx^{\frac{m+1}{2}}(xy-y+1)^{\frac{m-1}{2}}>0$ whenever $y^{m-1}>x^{2m}$ is proven for $m=2$ and $m=3$ in Appendix~\ref{Lemmata} for the proofs of the previous theorems; but we did not prove it here for all $m \in \mathbb{Z}_{+} \setminus \{1\}$. This assertion is actually already familiar to us, since it is asymptotically equivalent to Winterhof's result. Recall Winterhof proved that \[ q> (k-1)^{\frac{2m}{m-1}}  \implies \Bbb F_q \subseteq mR_k.\] He used Jacobi sums to deduce his bound, see \cite{Winterhof1}. 
\end{remark}
Let the notation stay the same with the last two theorems. Then, we have another nice result: 
\begin{thm}[Analog of S\'ark\"ozy's Theorem in Finite Fields] 
Let $k$ be a positive integer. If $E$ is a subset of $\Bbb F_q$ with size $|E| > \dfrac{qk}{\sqrt{q-1}}$ then it contains at least two distinct elements whose difference is a $k^{\text{\tiny th}}$ power. Thus, in particular if $k$ is fixed and $|E|=\Omega \left(\sqrt{q} \right)$ as $q \rightarrow \infty$, then $E$ contains at least two distinct elements whose difference is a $k^{\text{\tiny th}}$ power.
\end{thm}

\begin{proof}
As in the last two theorems, again we want to apply the spectral gap theorem.
Our calculations in the proof of Theorem~\ref{Main Theorem} showed that $n_{\ast} \leqslant \dfrac{q \sqrt{q- |R_k^{\ast} |} }{ |R_k^{\ast}|}$. 
We have
\[\frac{q \sqrt{q- |R_k^{\ast} |} }{ |R_k^{\ast}|} = \frac{q \sqrt{q- \frac{q-1}{d}} }{\frac{q-1}{d}}=\frac{q}{q-1}\sqrt{d} \sqrt{qd-q+1}. \] 
Since $d= \gcd(k,q-1) \leqslant q-1$ i.e. $d+1 \leqslant q$ we have 
\[ \frac{q}{q-1}\sqrt{d} \sqrt{qd-q+1} \leqslant \frac{q}{q-1}\sqrt{d} \sqrt{d(q-1)} = \frac{qd}{\sqrt{q-1}} \leqslant \frac{qk}{\sqrt{q-1}}.\]
If $\frac{qk}{\sqrt{q-1}}< |E|$, by Theorem~\ref{SGT} there is an edge from a vertex in $E$ to a distinct vertex in $E$ and the result follows.
\end{proof}

For example consider $\Bbb F_{41^2}=\Bbb F_{1681}$. If we take an arbitrary subset of $\Bbb F_{1681}$ with size greater than or equal to $124$ ($165$, $206$ resp.), then we know there exists at least two distinct elements in this set whose difference is a $3^{\text{\tiny rd}}$ ($4^{\text{\tiny th}}$, $5^{\text{\tiny th}}$ resp.) power. 

We claimed before that we will state the precise values of $\gamma(k,q)$ for every $4 \leqslant k \leqslant 37$ only except some of $\gamma(k,q) \leqslant 3$. Hence let $4 \leqslant k \leqslant 37$ be given. First note that by the theory, we know that for every $k$, there exists only a finite number of finite fields which are not coverable with the $k^{\text{\tiny th}}$ powers, see e.g. \cite{Torn}. Table~\ref{table1} lists all of the uncoverable finite fields for $4 \leqslant k \leqslant 37$. If we are given with a specific $q$ as well, then we can determine if $\gamma(k,q)=1$ using the criteria in Proposition~\ref{gamma1 case}. Theorem~\ref{table2-thm} and \ref{table3-thm} (upcoming) handle all of the cases where $ \gamma(k,q) =4,5,6$ and some of the cases where $ \gamma(k,q) =3$. Theorem~\ref{gamma bigger than or equal to 7-thm} provides all of the fields such that $ \gamma(k,q) \geqslant7$. Therefore, if we pick a $4 \leqslant k \leqslant 37$ and a coverable field $\Bbb F_{q}$ by taking into consideration Table~\ref{table1}, and if it doesn't fall into one of the categories in Theorem~\ref{Main Theorem},~\ref{table2-thm},~\ref{table3-thm} or ~\ref{gamma bigger than or equal to 7-thm}, then we can conclude $\gamma(k,q) \leqslant 3$.

The results in Proposition~\ref{table1-prop} below are discovered via exhaustive search using Sage, in principle one can calculate these with Theorem~G in \cite{Bhaskaran} doing finitely many calculations.

\begin{prop} \label{table1-prop}
Table~\ref{table1} lists all of the uncoverable fields for any $4 \leqslant k \leqslant 37$.
\end{prop}

\begin{table}[h!]
\centering
\caption{Uncoverable Fields}
\begin{tabu}  { | p{0.5cm} | | p{4cm} | }
\hline
\boldmath$k$ & \textbf{uncoverable fields for} \boldmath$k$ \\
\hline
$4$  & $9$  \\
\hline
$5$  &  $16$  \\
\hline
$6$  &  $4, 25$  \\
\hline
$7$  &  $8$  \\
\hline
$8$  &  $9, 49$  \\
\hline
$9$  &  $4, 64$ \\
\hline
$10$  &  $16, 81$  \\
\hline
$12$  &  $4, 9, 25, 121$  \\
\hline
$13$  &  $27$  \\
\hline
$14$  &  $8, 169$  \\
\hline
$15$  &  $4, 16$  \\
\hline
$16$  &  $9, 49$  \\
\hline
$17$  &  $256$  \\
\hline
$18$  &  $4, 25, 64, 289$  \\
\hline
\end{tabu}
\qquad  \ \ \ \ \ \ \ \ \ \ \ \ \ \ \ \ \ \ \ \ \ \ \ \ \ \ \ \
\begin{tabu}  { | p{0.5cm} | | p{4cm} | }
\hline
\boldmath$k$ & \textbf{uncoverable fields for} \boldmath$k$ \\
\hline
$20$  &  $9, 16, 81, 361$  \\
\hline
$21$  &  $4, 8, 64$  \\
\hline
$24$  &  $4, 9, 25, 49, 121, 529$  \\
\hline
$25$  &  $16$  \\
\hline
$26$  &  $27, 625$  \\
\hline
$27$  &  $4, 64$  \\
\hline
$28$  &  $8, 9, 169, 729$ \\
\hline
$30$  &  $4, 16, 25, 81, 841$  \\
\hline
$31$  &  $32, 125$  \\
\hline
$32$   &  $9, 49, 961$  \\
\hline
$33$   &  $4, 1024$  \\
\hline
$34$   &  $256$  \\
\hline
$35$   &  $8, 16$   \\
\hline
$36$   &  $4, 9, 25, 64, 121, 289$  \\
\hline
\end{tabu}  \newline
\label{table1}
\end{table} 

Notice that $k=11, 19, 22, 23, 29, 37$ cases are not on the Table~\ref{table1}. It is because there are not any uncoverable fields corresponding to those $k$ values.

\begin{thm} \label{table2-thm} 
For every integer $k$ satisfying $4 \leqslant k \leqslant 19$, Table~\ref{table2} lists all of the fields with size $q$ such that $\gamma(k,q)=3,4,5$ or $6$. 
\end{thm}
 
\begin{longtable}{ | p{0.5cm} | p{9cm} | p{2.5cm} | p{2.1cm} | p{2.1cm} | }
\caption{Fields} \label{table2} \\
\hline
\boldmath$k$ & \boldmath$\gamma(k,q)= 3$ & \boldmath$\gamma(k,q)= 4$ & \boldmath$\gamma(k,q)= 5$ & \boldmath$\gamma(k,q)= 6$ \\
\hline
\endfirsthead
\caption[]{(continued)} \\
\hline
\boldmath$k$ & \boldmath$\gamma(k,q)= 3$ & \boldmath$\gamma(k,q)= 4$ & \boldmath$\gamma(k,q)= 5$ & \boldmath$\gamma(k,q)= 6$ \\
\hline
\endhead
$4$ & $13$, $17$, $25$, $29$ & $5$ & \ \ & \ \  \\
\hline
$5$ & $31$, $41$, $61$ & \ \  & $11$ & \ \  \\
\hline
$6$ & $37$, $43$, $49$, $61$, $67$, $73$, $79$, $109$, $139$, $223$ & $19$, $31$ & \ \ & $7$, $13$ \\
\hline
$7$ & $64$, $71$, $113$, $127$ & $29$, $43$ & \ \ & \ \  \\
\hline
$8$ & $13$, $29$, $73$, $81$, $89$, $97$, $113$, $121$, $137$, $233$, $257$, $289$, $337$, $761$ & $5$, $25$, $41$ & \ \  & \ \  \\
\hline
$9$ & $7$, $73$, $109$, $127$, $163$, $181$, $199$, $271$, $307$, $343$ &  \ \  &  $37$ & \ \  \\
\hline
$10$ & $71$, $101$, $121$, $131$, $151$, $181$, $191$, $211$, $241$, $251$, $271$, $281$, $311$, $331$, $401$, $421$, $431$, $461$, $491$, $641$, $911$ & $41$, $61$ & $31$ & \ \  \\
\hline
$11$ & $199$, $331$, $353$, $419$, $463$, $617$ & $89$ & $67$ & \ \  \\
\hline
$12$ & $17$, $29$, $43$, $67$, $79$, $139$, $157$, $169$, $181$, $193$, $223$, $241$, $277$, $289$, $313$, $337$, $349$, $361$, $373$, $397$, $409$, $421$, $433$, $457$, $541$, $577$, $625$, $661$, $673$, $733$, $841$, $877$, $1069$, $1453$, $1669$, $1741$ & $5$, $19$, $31$, $61$, $97$, $109$, $229$ &  $73$ & $7$,  $37$, $49$ \\
\hline
$13$ & $131$, $157$, $313$, $443$, $521$, $547$, $599$, $677$, $859$, $911$, $937$, $1171$ & \ \ & $79$ & $53$ \\
\hline
$14$ & $64$, $197$, $211$, $239$, $281$, $337$, $379$, $421$, $449$, $463$, $491$, $547$, $617$, $631$, $659$, $673$, $701$, $743$, $757$, $827$, $911$, $953$, $967$, $1009$, $1499$, $2927$ & $71$, $113$, $127$ & \ \  & $43$  \\
\hline
$15$ & $7$, $41$, $151$, $181$, $211$, $241$, $256$, $271$, $331$, $361$, $421$, $541$, $571$, $601$, $631$, $661$, $691$, $751$, $811$, $961$, $991$, $1021$, $1051$, $1171$, $1201$, $3181$ & $121$ & $11$, $61$ & \ \ \\
\hline
$16$ & $13$, $29$, $73$, $89$, $121$, $137$, $233$, $241$, $289$, $337$, $353$, $401$, $433$, $449$, $529$, $577$, $593$, $625$, $641$, $673$, $761$, $769$, $881$, $929$, $977$, $1009$, $1201$, $1249$, $1297$, $1409$, $1489$, $1697$, $2017$, $2401$, $2593$ & $5$, $25$, $41$, $81$, $113$, $193$, $257$ & \ \ & $97$ \\
\hline
$17$ & $239$, $307$, $409$, $443$, $613$, $647$, $919$, $953$, $1021$, $1123$, $1259$, $1327$, $1361$, $1531$, $1667$ & $137$ & \ \ & $103$ \\
\hline
$18$ & $43$, $49$, $61$, $67$, $79$, $139$, $223$, $307$, $343$, $361$, $379$, $433$, $487$, $523$, $541$, $577$, $613$, $631$, $739$, $757$, $811$, $829$, $883$, $919$, $937$, $991$, $1009$, $1063$, $1117$, $1153$, $1171$, $1279$, $1369$, $1423$, $1459$, $1693$, $1747$, $2089$, $2197$, $2251$, $2269$, $2287$, $2503$, $2719$, $3259$, $4519$ & $31$, $163$, $181$, $271$, $397$ & $199$ & $7$, $13$, $109$, $127$ \\
\hline
$19$ & $343$, $457$, $571$, $647$, $761$, $1103$, $1217$, $1483$, $1559$, $1597$, $1787$, $2053$, $2129$, $2357$, $2927$ & $191$, $229$, $419$ & \ \ & \ \  \\
\hline
\end{longtable}

\begin{thm} \label{table3-thm} 
Let $20 \leqslant k \leqslant 37$ be any given integer. 
\vspace{-0.3cm}
\begin{itemize}
\item Table~\ref{table3} lists all of the fields with size $q$ such that $ \gamma(k,q)=4,5$ or $6$. 
\vspace{-0.3cm}
\item Table~\ref{table3} lists all of the fields with size $q$ such that $ \gamma(k,q)=3$ and $q \leqslant (k-1)^3$.
\end{itemize}
\vspace{-0.3cm}
(We do not list all of the finite fields with $ \gamma(k,q)=3$ here, since when we ran the computer, we used a bound from \cite{Winterhof1}, which is if $q >(k-1)^3$ then $\Bbb F_q \subseteq 3R_{k}$.)
\end{thm}
 
\begin{longtable}{ | p{0.5cm} | p{9cm} | p{2.5cm} | p{2.1cm} | p{2.1cm} | }
\caption{Fields} \label{table3} \\
\hline
\boldmath$k$ & \boldmath$\gamma(k,q)= 3$ & \boldmath$\gamma(k,q)= 4$ & \boldmath$\gamma(k,q)= 5$ & \boldmath$\gamma(k,q)= 6$ \\
\hline
\endfirsthead
\caption[]{(continued)} \\
\hline
\boldmath$k$ & \boldmath$\gamma(k,q)= 3$ & \boldmath$\gamma(k,q)= 4$ & \boldmath$\gamma(k,q)= 5$ & \boldmath$\gamma(k,q)= 6$ \\
\hline
\endhead
$20$ & $13$, $17$, $25$, $29$, $71$, $131$, $151$, $191$, $211$, $251$, $271$, $281$, $311$, $331$, $401$, $421$, $431$, $461$, $491$, $541$, $601$, $641$, $661$, $701$, $761$, $821$, $841$, $881$, $911$, $941$, $961$, $1021$, $1061$, $1181$, $1201$, $1301$, $1321$, $1361$, $1381$, $1481$, $1601$, $1621$, $1681$, $1721$, $1741$, $1861$, $1901$, $2221$, $2281$, $2381$, $2621$, $2861$ & $5$, $181$, $241$, $521$ & $31$, $101$ & \ \  \\
\hline
$21$ & $7$, $71$, $113$, $379$, $421$, $463$, $547$, $631$, $673$, $757$, $841$, $883$, $967$, $1009$, $1051$, $1093$, $1303$, $1429$, $1471$, $1597$, $1723$, $1849$, $2143$, $2437$, $2521$, $2647$, $2689$, $2857$, $3319$, $5167$ & $29$, $169$, $211$, $337$ & \ \ & $127$ \\
\hline
$22$ & $243$, $353$, $397$, $529$, $617$, $661$, $683$, $727$, $859$, $881$, $947$, $991$, $1013$, $1123$, $1277$, $1321$, $1409$, $1453$, $1607$, $1783$, $1871$, $2003$, $2069$, $2113$, $2179$, $2267$, $2333$, $2377$, $2399$, $2531$, $2663$, $2729$, $2927$, $3037$, $3719$, $3851$, $3917$ & $199$, $331$, $419$, $463$ & \ \ & \ \ \\
\hline
$23$ & $461$, $599$, $691$, $829$, $967$, $1013$, $1151$, $1289$, $1381$, $1427$, $1657$, $1933$, $1979$, $2209$, $2347$, $2393$, $2531$, $3083$, $3313$, $4831$, $4969$, $5659$, $8419$ & $277$ & \ \ & \ \ \\
\hline
$24$ & $29$, $43$, $67$, $79$, $81$, $89$, $113$, $137$, $139$, $157$, $181$, $223$, $233$, $257$, $277$, $337$, $349$, $373$, $397$, $421$, $433$, $541$, $601$, $625$, $661$, $673$, $733$, $761$, $769$, $841$, $877$, $937$, $961$, $1009$, $1033$, $1069$, $1129$, $1153$, $1201$, $1249$, $1297$, $1321$, $1369$, $1453$, $1489$, $1609$, $1657$, $1669$, $1681$, $1741$, $1753$, $1777$, $1801$, $1849$, $1873$, $1993$, $2017$, $2089$, $2113$, $2137$, $2161$, $2281$, $2377$, $2401$, $2473$ & $5$, $19$, $31$, $41$, $61$, $109$, $169$, $229$, $241$, $313$, $361$, $409$, $457$, $577$ & $193$, $289$ & $7$, $37$ \\
\hline
$25$ & $31$, $41$, $61$, $401$, $601$, $701$, $751$, $1051$, $1151$, $1201$, $1301$, $1451$, $1601$, $1801$, $1901$, $1951$, $2251$, $2351$, $2551$, $2801$, $2851$, $3001$, $3251$, $3301$, $3851$, $5051$, $5501$ & $251$ & $11$ & \ \ \\
\hline
$26$ & $521$, $677$, $729$, $859$, $911$, $937$, $1093$, $1171$, $1223$, $1249$, $1301$, $1327$, $1483$, $1613$, $1847$, $1873$, $1951$, $2003$, $2029$, $2081$, $2237$, $2341$, $2393$, $2549$, $2809$, $2861$, $2887$, $2939$, $3121$, $3329$, $3407$, $3433$, $3511$, $3719$, $3797$ & $313$, $443$, $547$, $599$ & \ \ & $131$ \\
\hline
$27$ & $7$, $73$, $127$, $181$, $199$, $307$, $343$, $379$, $757$, $811$, $919$, $1297$, $1459$, $1567$, $1621$, $1783$, $1999$, $2053$, $2161$, $2269$, $2377$, $2539$, $2593$, $2647$, $2917$, $2971$, $3079$, $3187$, $3457$, $3727$, $3889$, $3943$, $4051$, $4159$, $4483$, $4591$, $4861$, $4969$, $5023$, $6859$, $7993$, $9397$, $9829$ & $271$, $433$, $487$, $541$ & $37$ & \ \ \\
\hline
$28$ & $13$, $17$, $25$, $64$, $211$, $239$, $379$, $463$, $491$, $547$, $631$, $659$, $673$, $701$, $743$, $757$, $827$, $841$, $911$, $953$, $967$, $1009$, $1093$, $1289$, $1373$, $1429$, $1499$, $1597$, $1681$, $1709$, $1849$, $1877$, $1933$, $2017$, $2129$, $2213$, $2269$, $2297$, $2381$, $2437$, $2521$, $2549$, $2633$, $2689$, $2801$, $2857$, $2927$, $2969$, $3109$, $3137$, $3221$, $3361$, $3389$, $3529$, $3557$, $3613$, $3697$, $4201$, $4229$, $4397$, $4481$, $4621$, $4649$, $4733$, $4789$, $4817$, $4957$ & $5$, $71$, $127$, $337$, $421$, $449$, $617$ & $197$, $281$ & $43$ \\
\hline
$29$ & $523$, $929$, $1103$, $1277$, $1451$, $1567$, $1741$, $1973$, $2089$, $2437$, $3191$, $3307$, $3481$, $3539$, $4003$, $4177$, $4409$, $4583$, $4931$, $5569$, $5801$, $6091$, $8237$ & $349$ & $233$ & \ \ \\
\hline
$30$ & $37$, $43$, $49$, $67$, $71$, $73$, $79$, $101$, $109$, $131$, $139$, $191$, $223$, $251$, $256$, $281$, $311$, $401$, $431$, $461$, $491$, $601$, $631$, $641$, $691$, $751$, $811$, $911$, $961$, $991$, $1021$, $1051$, $1171$, $1201$, $1231$, $1291$, $1321$, $1381$, $1471$, $1531$, $1621$, $1681$, $1741$, $1801$, $1831$, $1861$, $1951$, $2011$, $2131$, $2161$, $2221$, $2251$, $2281$, $2311$, $2341$, $2371$, $2401$, $2521$, $2551$, $2671$, $2731$, $2791$, $2851$, $2971$, $3001$, $3061$, $3121$, $3181$, $3271$, $3301$, $3331$, $3361$, $3391$, $3511$, $3541$, $3571$, $3631$, $3691$, $3721$, $3931$, $4021$, $4051$, $4111$, $4201$, $4261$, $4441$, $4561$, $4831$, $4861$ & $19$, $41$, $271$, $331$, $361$, $421$, $541$, $571$, $661$ & $211$, $241$ & $7$, $13$, $151$ \\
\hline
$31$ & $683$, $1024$, $1117$, $1303$, $1427$, $1489$, $1613$, $1861$, $2357$, $2543$, $2729$, $2791$, $3163$, $3659$, $3907$, $4093$, $4217$, $4651$, $5023$, $5147$, $5209$, $5519$, $5581$, $5953$, $7193$ & $311$ & $373$ & \ \ \\
\hline
$32$ & $13$, $29$, $73$, $89$, $121$, $137$, $233$, $241$, $337$, $401$, $433$, $529$, $593$, $625$, $641$, $673$, $761$, $881$, $929$, $977$, $1009$, $1153$, $1201$, $1217$, $1249$, $1297$, $1409$, $1489$, $1601$, $1697$, $1889$, $2017$, $2081$, $2113$, $2209$, $2273$, $2401$, $2593$, $2657$, $2689$, $2753$, $3041$, $3137$, $3169$, $3329$, $3361$, $3457$, $3617$, $4001$, $4129$, $4289$, $4513$, $4673$, $4801$, $4993$ & $5$, $25$, $41$, $81$, $113$, $289$, $353$, $449$, $577$, $769$ & \ \ & $257$ \\
\hline
$33$ & $7$, $353$, $419$, $617$, $859$, $991$, $1123$, $1321$, $1453$, $1783$, $1849$, $2113$, $2179$, $2311$, $2377$, $2707$, $2971$, $3037$, $3169$, $3301$, $3433$, $3499$, $3631$, $3697$, $4027$, $4093$, $4159$, $4357$, $4423$, $4489$, $4621$, $4951$, $5281$, $5347$, $5413$, $5479$, $5743$, $6007$, $6073$, $6271$, $6337$, $6469$, $6997$, $7723$, $7789$, $7921$, $8317$, $8647$, $9439$ & $89$, $397$, $463$, $529$, $661$, $727$ & $331$ & \ \ \\
\hline
$34$ & $919$, $1021$, $1123$, $1259$, $1327$, $1361$, $1429$, $1531$, $1667$, $1871$, $1973$, $2143$, $2347$, $2381$, $2551$, $2687$, $2789$, $2857$, $3061$, $3163$, $3299$, $3469$, $3571$, $3673$, $3877$, $3911$, $4013$, $4217$, $4421$, $4523$, $4591$, $4931$, $4999$, $5101$, $5237$, $5407$, $5441$, $5849$ & $443$, $613$, $647$, $953$ & $307$, $409$ & $239$ \\
\hline
$35$ & $31$, $41$, $61$, $64$, $113$, $127$, $701$, $841$, $911$, $1051$, $1331$, $1471$, $1681$, $2311$, $2381$, $2521$, $2591$, $2731$, $2801$, $3011$, $3221$, $3361$, $3571$, $3851$, $4096$, $4201$, $4271$, $4481$, $4621$, $4691$, $4831$, $5041$, $5531$, $5741$, $5881$, $6301$, $6581$, $6791$, $7001$, $7351$, $7561$, $8191$, $8681$, $8821$ & $29$, $43$, $421$, $491$, $631$ & $11$, $281$ & \ \ \\
\hline
$36$ & $17$, $29$, $43$, $67$, $79$, $139$, $157$, $169$, $193$, $223$, $241$, $277$, $307$, $313$, $337$, $343$, $349$, $373$, $379$, $409$, $421$, $457$, $487$, $523$, $625$, $631$, $661$, $673$, $733$, $739$, $811$, $829$, $841$, $877$, $883$, $919$, $937$, $991$, $1009$, $1063$, $1069$, $1153$, $1171$, $1279$, $1297$, $1369$, $1423$, $1453$, $1459$, $1549$, $1621$, $1657$, $1669$, $1693$, $1741$, $1747$, $1801$, $1873$, $2017$, $2053$, $2089$, $2161$, $2197$, $2251$, $2269$, $2287$, $2341$, $2377$, $2503$, $2521$, $2557$, $2593$, $2719$, $2809$, $2917$, $2953$, $3061$, $3169$, $3259$, $3313$, $3457$, $3529$, $3637$, $3673$, $3709$, $3853$, $3889$, $4177$, $4357$, $4519$, $4789$, $4861$, $4933$, $4969$ & $5$, $31$, $61$, $97$, $163$, $229$, $271$, $433$, $541$, $577$, $613$, $757$, $1117$ & $199$, $361$, $397$ & $7$, $49$, $127$ \\
\hline
$37$ & $1259$, $1481$, $1777$, $1999$, $2221$, $2591$, $2887$, $3109$, $3257$, $3331$, $3701$, $3923$, $4219$, $4441$, $4663$, $5107$, $5477$, $6143$, $6217$, $6661$, $7253$, $7549$, $7919$, $7993$, $8363$, $8807$, $9103$, $9473$ & $593$ & \ \ & \ \  \\
\hline
\end{longtable}

\begin{thm} \label{gamma bigger than or equal to 7-thm}
For $4 \leqslant k \leqslant 37$, the following is the list of all of the fields with size $q$ such that $ \gamma (k,q) \geqslant 7:$
\begin{itemize}
	\item $\gamma(8,17)=8$ \vspace{-0.3cm}
	\item $\gamma(9,19)=9$ \vspace{-0.3cm}
	\item $ \gamma(10,11)=10$ \vspace{-0.3cm}
	\item $\gamma(11,23)=11$ \vspace{-0.3cm}
	\item $\gamma(12,13)=12$ \vspace{-0.3cm}
	\item $ \gamma(14,29)=14 $ \vspace{-0.3cm}
	\item $ \gamma(15,31)=15 $ \vspace{-0.3cm}
	\item $ \gamma(16,17)=16 $ \vspace{-0.3cm}
	\item $ \gamma(18,73)=7$, \  $ \gamma(18,19)=18$,\  $ \gamma(18,37)=18 $ \vspace{-0.3cm}
	\item $ \gamma(20,121)=7 $, \ $ \gamma(20,61)=8 $, \ $ \gamma(20,11)=10 $, \ $ \gamma(20,41)=20 $ \vspace{-0.3cm}
	\item $ \gamma(21,43)=21 $ \vspace{-0.3cm}
	\item $ \gamma(22,89)=7 $, \ $ \gamma(22,67)=8 $, \ $ \gamma(22,23)=22 $ \vspace{-0.3cm}
	\item $ \gamma(23,139)=7 $, \ $ \gamma(23,47)=23 $ \vspace{-0.3cm}
	\item $ \gamma(24,17)=8 $, \ $ \gamma(24,73)=8 $, \ $ \gamma(24,97)=8 $, \ $ \gamma(24,13)=12 $ \vspace{-0.3cm}
	\item $ \gamma(25,151)=7 $, \ $ \gamma(25,101)=9 $ \vspace{-0.3cm}
	\item $ \gamma(26,157)=8 $, \ $ \gamma(26,79)=9 $, \ $ \gamma(26,53)=26 $ \vspace{-0.3cm}
	\item $ \gamma(27,163)=8 $, \ $ \gamma(27,19)=9 $, \ $ \gamma(27,109)=9 $ \vspace{-0.3cm}
	\item $ \gamma(28,113)=7 $, \ $ \gamma(28,29)=28 $ \vspace{-0.3cm}
	\item $ \gamma(29,59)=29 $ \vspace{-0.3cm}
	\item $ \gamma(30,181)=8 $, \ $ \gamma(30,11)=10 $, \ $ \gamma(30,121)=10 $, \ $ \gamma(30,31)=30 $, \ $\gamma(30,61)=30 $ \vspace{-0.3cm}
	\item $ \gamma(32,193)=8 $, \ $ \gamma(32,97)=10 $, \ $ \gamma(32,17)=16 $ \vspace{-0.3cm}
	\item $ \gamma(33,199)=9 $, \ $ \gamma(33,23)=11 $, \ $ \gamma(33,67)=33 $ \vspace{-0.3cm}
	\item $ \gamma(34,103)=10 $, \ $ \gamma(34,137)=10 $ \vspace{-0.3cm}
	\item $ \gamma(35,211)=9 $, \ $ \gamma(35,71)=35 $ \vspace{-0.3cm}
	\item $ \gamma(36,181)=7 $, \ $ \gamma(36,109)=11 $, \ $ \gamma(36,13)=12 $, \ $ \gamma(36,19)=18 $, \ $ \gamma(36,37)=36 $, \\
$ \gamma(36,73)=36 $ \vspace{-0.3cm}
	\item $ \gamma(37,149)=9 $, \ $ \gamma(37,223)=9 $.
\end{itemize}
\end{thm}

Let $k$ be any positive integer. Then, we define $\gamma(k):= \max \gamma(k,q)$ where we take the maximum value of $\gamma(k,q)$ over the set $\{ q \mid \gamma(k,q) \text{ exists} \}$. Hence, if a finite field is coverable with the $k^{\text{\tiny th}}$ powers, every element of the field can be written as a sum of $\gamma(k)$ many $k^{\text{\tiny th}}$ powers. By Proposition~\ref{gamma1 case} we have $\gamma(k) \leqslant k$. The following proposition gives us the exact values of $\gamma(k)$ for $k \leqslant 37$:

\begin{prop} 
Let $k \leqslant 37$. If $\gamma(k)$ is not provided in Table~\ref{table5}, that means $\gamma(k)=k$.
\end{prop}

\begin{table}[h!]
\centering
\caption{$\gamma(k)$ values}
\begin{tabu}  {| p{2.5cm} | p{2.5cm} | p{2.5cm} | p{2.5cm} | p{2.5cm} | p{2.5cm} | }
\hline
$\gamma(7)=4$ & $\gamma(13)=6$ & $\gamma(17)=6$ & $\gamma(19)=4$ & $\gamma(24)=12$ & $\gamma(25)=9$  \\
\hline
$\gamma(27)=9$ & $\gamma(31)=5$ & $\gamma(32)=16$ & $\gamma(34)=10$ & $\gamma(37)=9$ & \\
\hline
\end{tabu}	 \newline
\label{table5}
\end{table} 

\section{Waring's Problem in Matrix Rings} \label{matrix section}

Let $\Bbb F_q$ be the finite field with $q$ elements and let $p$ denote the characteristic of $\Bbb F_q$. Also let $k$ be a positive integer and $\operatorname{Mat}_n(\Bbb F_q)$ be the ring of $n \times n$ matrices over $\Bbb F_q$. In this section we want to find an integer $m$ with the property that every element of $\operatorname{Mat}_n(\Bbb F_q)$ can be written as a sum of at most $m$ many $k^{\text{\tiny th}}$ powers in $\operatorname{Mat}_n(\Bbb F_q)$. First, we need to develop our machinery. Take any $A \in \operatorname{Mat}_n(\Bbb F_q)$. Let $P(x) \in \Bbb F_q[x]$ be any polynomial over $\Bbb F_q$, then $P(x) = a_{0}+a_{1}x+\cdots+a_{l}x^{l}$ for some $a_0, a_1 \cdots a_l \in \Bbb F_q$. We can evaluate any polynomial $P(x)$ at $A$; for example here $P(A) =a_{0}I+a_{1}A+\cdots+a_{l}A^{l}$ where $I$ denotes the $n \times n$ identity matrix over $\Bbb F_q$. Then, we define $\Bbb F_q [A] = \{P(A) \mid P(x) \in \Bbb F_q[x]\}$. We want to find an $m$ such that $A=B_1^k+ \cdots +B_m^k$ for some $B_1, \cdots, B_m \in \operatorname{Mat}_n(\Bbb F_q)$. Instead of working with $\operatorname{Mat}_n(\Bbb F_q)$ we will work with $\Bbb F_q[A]$, meaning we will try to write $A$ as a sum of $k^{\text{\tiny th}}$ powers using the elements of $\Bbb F_q[A]$. This is more convenient since we have more tools to use in $\Bbb F_q[A]$, but we are only able to provide upper bounds for $m$ with this approach, rather than exact values.
Define
\begin{align*} 
\varphi: \Bbb F_q[x] &\longrightarrow \Bbb F_q[A]\\
a_0+a_{1}x+\cdots+a_{l}x^{l} &\longmapsto a_{0}I+a_{1}A+\cdots+a_{l}A^{l}.
\end{align*}
Notice that \[ \ker \varphi = \left\{ P(x) \in \Bbb F_q[x] \text{ s.t. } P(A)=0_{n \times n} \right\}=\langle m_{A}(x) \rangle \] where $m_{A}(x)$ denotes the minimum polynomial of $A$, and $\operatorname{Im} \varphi =\Bbb F_q[A]$. By first isomorphism theorem we have \[ \bigslant{\Bbb F_q[x]}{\langle m_{A}(x) \rangle} \cong \Bbb F_q[A].\]

Let's assume $m_{A}(x) = p_{1}(x)^{i_{1}} \cdots p_{j}(x)^{i_{j}}$ for some $ p_1,p_2, \cdots, p_j$ distinct irreducible polynomials over  $\Bbb F_q$. Using Chinese remainder theorem we have \[ \bigslant{\Bbb F_q[x]}{\langle p_1^{i_1} \cdots p_j^{i_j}\rangle} \cong \bigslant{\Bbb F_q[x]}{\langle p_1^{i_1} \rangle}\times \cdots \times \bigslant{\Bbb F_q[x]}{\langle p_j^{i_j} \rangle}.\] We need one more tool before we state the main result of the section.

\begin{prop} \label{fancy Henseling}
Let $f(x)$ be an irreducible polynomial over $\Bbb F_q$ with degree $n$. Let $i$ and $k$ be positive integers such that $p \nmid k$ where $p$ denotes the characteristic of $\Bbb F_q$.
\vspace{-0.3cm} 
\begin{itemize}
\item If every element of $\bigslant{\Bbb F_q[x]}{\langle f(x) \rangle}$ is a $k^{\text{\tiny th}}$ power, then the elements of $\bigslant{\Bbb F_q[x]}{\langle f^i(x) \rangle}$ can be written as a sum of two $k^{\text{\tiny th}}$ powers. 
\vspace{-0.3cm} 
\item Assume every element of $\bigslant{\Bbb F_q[x]}{\langle f(x) \rangle}$ can be written as a sum of $m$ many $k^{\text{\tiny th}}$ powers for some $m > 1$. If $-1$ is a $k^{\text{\tiny th}}$ power in $\Bbb F_q$, then the elements of $\bigslant{\Bbb F_q[x]}{\langle f^i(x) \rangle}$ can be written as a sum of $m$ many $k^{\text{\tiny th}}$ powers. If $-1$ is not a $k^{\text{\tiny th}}$ power in $\Bbb F_q$, then the elements of $\bigslant{\Bbb F_q[x]}{\langle f^i(x) \rangle}$ can be written as a sum of $m+1$ many $k^{\text{\tiny th}}$ powers.
\end{itemize}
\end{prop} 

\begin{proof} 
We denote $f(x)$ with $f$ in this proof. Take any element from $\bigslant{\Bbb F_q[x]}{\langle f^2 \rangle}$. Using division algorithm we can write an equivalence class representative of this element as $r_1+r_2f$ where $r_1,r_2 \in \Bbb F_q[x]$ and $\deg(r_1), \deg(r_2)< n$. It follows from the assumption that there exist some $B_1, B_2, \cdots, B_m \in \Bbb F_q[x]$ such that $r_1 \equiv B_1^k+B_2^k+ \cdots +B_m^k \pmod{f}$.

If some of $B_i $'s are zeros modulo $f$, we can erase them from the list of $ B_i $'s, and use the ones which are not zero modulo $f$ for the next step. We can erase them, because if $B_i \equiv 0  \pmod{f}$, $B_i^k \equiv 0  \pmod{f}$ so the summation of the $k^{\text{\tiny th}}$ powers of the reduced list will still be $r_1$. If all of the $B_i $'s are zeros modulo $f$, then we need to be a little bit careful. If $k$ is odd, or if $k$ is even and $-1$ is a $k^{\text{\tiny th}}$ power in $\Bbb F_q$, simply take $B_{i_1}=1$ and $B_{i_2}=-1$ and proceed to the next step. If $k$ is even and $-1$ is not a $k^{\text{\tiny th}}$ power in $\Bbb F_q$, then we can write $0 \equiv B_{i_1}^k+ (-1)  \pmod{f}$ for $B_{i_1}=1$, and we know $-1$ can be written as a sum of $m$ many $k^{\text{\tiny th}}$ powers by assumption. So we can assume $B_{i_1}, \cdots, B_{i_j} $ are all nonzero in $\bmod{f}$ and their summation is equivalent to $r_1$ in $\bmod{f}$.

Define the polynomial $Q(t)=t^k+B_{i_1}^k+  \cdots +B_{i_{j-1}}^k-r_1-r_2f$ over $\Bbb F_q[x]$, use Hensel's lemma for the polynomial rings (Proposition~\ref{Hensel for poly-strong version}) for $Q(t)$ to conclude there exists a $g_2 \in \Bbb F_q[x]$ such that $ g_2^k+B_{i_1}^k+ \cdots+ B_{i_{j-1}}^k= r_1+r_2f \pmod{f^2}$ and the result for $i=2$ follows. Here, we are allowed to use Hensel's lemma, since $B_{i_j}$ is a root of $Q(t)$ in $\bmod{f}$ and $\nu (Q'(B_{i_j}))=0 < \frac{1}{2}$. $\nu (Q'(B_{i_j}))=0$ since $Q'(t)=k t^{k-1}$ and $Q'(B_{i_j})= k B_{i_j}^{k-1} \not\equiv 0 \pmod{f}$. This follows from $ f \nmid k$, $ p \nmid k $, $f \nmid B_{i_j}$ and $f$ being irreducible.

Once we lift a solution from $\bigslant{\Bbb F_q[x]}{\langle f(x) \rangle}$ to $\bigslant{\Bbb F_q[x]}{\langle f^2(x) \rangle}$, it is easy. Third condition in Proposition~\ref{Hensel for poly-strong version} allows us to apply Hensel's lemma again and again, and get a solution in each $\bigslant{\Bbb F_q[x]}{\langle f^i(x) \rangle}$ for any $i \in \mathbb{Z}_{+}$.
\end{proof}

Using this proposition, the discussion right before the proposition and finite field tables in Section~\ref{fields section} we literally have dozens of nice results for the matrix rings, but it would be a waste of paper to write out all of them. That is why we will state the results only for some small $k$ values to get across our point to the reader in the following theorem. Note that we denote with $R_{k,n}$ the set of all $k^{\text{\tiny th}}$ powers in $\operatorname{Mat}_n(\Bbb F_q)$ i.e. $R_{k,n}= \left\{ M^k \ | \ M \in \operatorname{Mat}_n(\Bbb F_q) \right\}$. Also, $\operatorname{Mat}_n(\Bbb F_q) \subseteq m R_{k,n}$ means every element of $\operatorname{Mat}_n(\Bbb F_q)$ can be written as a sum of $m$ many $k^{\text{\tiny th}}$ powers.

\begin{thm}\label{matrixTable} 
In the following table in a fixed row, let $k$ be a positive integer as in the first column, and let $\Bbb F_q$ be the finite field with $q$ elements where $q$ is different than the values in the second column. Then $\operatorname{Mat}_n(\Bbb F_q) \subseteq m R_{k,n}$ for any $n$. In fact, in this case any $A \in \operatorname{Mat}_n(\Bbb F_q)$ has $A=B_{1}^k+B_{2}^k+\cdots+B_{m}^k$ with $B_{1}$, $B_{2}, \cdots, B_{m}$ in the subring generated by $A$. (For example the first row should be read as "If $|\Bbb F_q| \neq 2,4 $, or any power of $3$, then $ \operatorname{Mat}_n(\Bbb F_q) \subseteq 3 R_{3,n}$ for any $n$. In fact, in this case any $A \in \operatorname{Mat}_n(\Bbb F_q)$ has $A=B_{1}^3+B_{2}^3+B_{3}^3$ with $B_{1}$, $B_{2}$, $B_{3}$ in the subring generated by $A$.")
\end{thm}

\begin{longtable}{ | p{1.8cm} | p{12cm} | p{1.8cm} | }
\hline
\boldmath$k=$ &  \boldmath$q \neq$ & \boldmath$m=$\\
\hline \hline
\multirow{2}{*}{} $3$ & $2$, $4$, $3^s$ & $3$ \\ \cline{2-2} \cline{3-3} 
                  	         & $2$, $4$, $7$, $3^s$ &   $2$ \\ \cline{2-2} \cline{3-3} \hline \hline 
\multirow{3}{*}{} $4$ & $3$, $9$, $2^s$ &   $5$ \\ \cline{2-2} \cline{3-3} 
                  	         & $3$, $5$, $9$, $2^s$ &  $4$ \\ \cline{2-2} \cline{3-3} 
                  	         & $3$, $5$, $9$, $13$, $17$, $25$, $29$, $2^s$ &  $3$ \\ \cline{2-2}  \cline{3-3} \hline \hline
\multirow{3}{*}{} $5$ & $2$, $4$, $16$, $5^s$ & $5$ \\ \cline{2-2} \cline{3-3} 
                  	         & $2$, $4$, $11$, $16$, $5^s$ & $3$ \\ \cline{2-2} \cline{3-3} 
                  	         & $2$, $4$, $11$, $16$, $31$, $41$, $61$, $5^s$ & $2$ \\ \cline{2-2}  \cline{3-3} \hline \hline 	      
\multirow{4}{*}{} $6$ & $2$, $4$, $5$, $25$, $2^s$, $3^s$ & $7$ \\ \cline{2-2} \cline{3-3} 
					    & $2$, $4$, $5$, $7$, $13$, $25$, $2^s$, $3^s$ & $5$ \\ \cline{2-2} \cline{3-3} 
                  	         & $2$, $4$, $5$, $7$, $13$, $19$, $25$, $31$, $2^s$, $3^s$ & $4$ \\ \cline{2-2} \cline{3-3} 
                  	         & $2$, $4$, $5$, $7$, $13$, $19$, $25$, $31$, $37$, $43$, $49$, $61$, $67$, $73$, $79$, $109$, $139$, $223$, $2^s$, $3^s$ & $3$ \\  \cline{2-2}  \cline{3-3} \hline \hline
\multirow{3}{*}{} $7$ & $2$, $8$, $7^s$ & $4$  \\ \cline{2-2} \cline{3-3} 
                  	         & $2$, $8$, $29$, $43$, $7^s$ & $3$ \\ \cline{2-2} \cline{3-3} 
                  	         & $2$, $4$, $8$, $29$, $43$, $64$, $71$, $113$, $127$, $7^s$ & $2$ \\ \cline{2-2}  \cline{3-3} \hline \hline
\multirow{4}{*}{} $8$ & $3$, $7$, $9$, $49$, $2^s$ & $9$ \\ \cline{2-2} \cline{3-3} 
                  	       & $3$, $7$, $9$, $17$, $49$, $2^s$ & $5$ \\ \cline{2-2} \cline{3-3} 
	      			  & $3$, $5$, $7$, $9$, $17$, $25$, $41$, $49$, $2^s$ & $4$ \\ \cline{2-2}  \cline{3-3}
                  	      & $3$, $5$, $7$, $9$, $11$, $13$, $17$, $25$, $29$, $41$, $49$, $73$, $81$, $89$, $97$, $113$, $121$, $137$, $233$, $257$, $289$, $337$, $761$, $2^s$ & $3$ \\ \cline{2-2}  \cline{3-3} \hline \hline 
\multirow{4}{*}{} $9$ & $2$, $4$, $8$, $64$, $3^s$ & $9$  \\ \cline{2-2} \cline{3-3} 
                  	       & $2$, $4$, $8$, $19$, $64$, $3^s$ & $5$ \\ \cline{2-2} \cline{3-3} 
	      			  & $2$, $4$, $8$, $19$, $37$, $64$, $3^s$ & $3$ \\  \cline{2-2} \cline{3-3} 
                  	       & $2$, $4$, $7$, $8$, $19$, $37$, $64$, $73$, $109$, $127$, $163$, $181$, $199$, $271$, $307$, $343$, $3^s$  & $2$ \\ \cline{2-2}  \cline{3-3} \hline \hline      
\multirow{5}{*}{} $10$ & $2$, $3$, $4$, $9$, $16$, $81$, $2^s$, $5^s$ & $11$ \\ \cline{2-2} \cline{3-3} 
                  	           & $2$, $3$, $4$, $9$, $11$, $16$, $81$, $2^s$, $5^s$ & $6$ \\ \cline{2-2} \cline{3-3} 
	      			      & $2$, $3$, $4$, $9$, $11$, $16$, $31$, $81$, $2^s$, $5^s$ & $5$ \\ \cline{2-2} \cline{3-3} 
				 	      & $2$, $3$, $4$, $9$, $11$, $16$, $31$, $41$, $61$, $81$, $2^s$, $5^s$ & $4$ \\ \cline{2-2} \cline{3-3} 
                  	           & $2$, $3$, $4$, $9$, $11$, $16$, $31$, $41$, $61$, $71$, $81$, $101$, $121$, $131$, $151$, $181$, $191$, $211$, $241$, $251$, $271$, $281$, $311$, $331$, $401$, $421$, $431$, $461$, $491$, $641$, $911$, $2^s$, $5^s$ & $3$ \\ \cline{2-2}  \cline{3-3} \hline \hline
\multirow{4}{*}{} $11$ & $23$, $11^s$ & $5$ \\ \cline{2-2} \cline{3-3} 
                  	           & $23$, $67$, $11^s$ & $4$ \\ \cline{2-2} \cline{3-3} 
	      			      & $23$, $67$, $89$, $11^s$ & $3$ \\ \cline{2-2} \cline{3-3} 
				 	      & $23$, $67$, $89$, $199$, $331$, $353$, $419$, $463$, $617$, $11^s$ & $2$ \\ \cline{2-2} \cline{3-3} \hline
\end{longtable}

\begin{proof}
To explain the process, we will only do for the first row. Let $k=3$. By Section~\ref{fields section} we know that every element of a finite field can be written as a sum of three cubes as long as $q \neq 4$. Let $A \in \operatorname{Mat}_n(\Bbb F_q)$. By the discussion prior to Proposition~\ref{fancy Henseling}, we know that $\Bbb F_q[A] \cong  \bigslant{\Bbb F_q[x]}{\langle p_1^{i_1} \rangle} \times \cdots \times \bigslant{\Bbb F_q[x]}{\langle p_j^{i_j} \rangle}$ for some $p_1,p_2, \cdots, p_j$ distinct irreducible polynomials over $\Bbb F_q$. We first exclude all of the subfields of $\Bbb F_4$ from our list so that none of the $\bigslant{\Bbb F_q[x]}{\langle p_r\rangle}$ will be isomorphic to $\Bbb F_4$. So, $\bigslant{\Bbb F_q[x]}{\langle p_r \rangle}$ will be different than $\Bbb F_4$, as long as $\Bbb F_q \neq \Bbb F_2, \Bbb F_4$. This guarantees that every element of $\bigslant{\Bbb F_q[x]}{\langle p_r \rangle}$ can be written as a sum of three cubes. Then, we want to lift the $k^{\text{\tiny th}}$ powers from $\bigslant{\Bbb F_q[x]}{\langle p_r \rangle}$ to $\bigslant{\Bbb F_q[x]}{\langle p_r^{i_r} \rangle}$. We will use Proposition~\ref{fancy Henseling}. To use this proposition we need the characteristic of the finite field not to divide $k$. That is why we need to exclude $p=3$ cases. Once we guarantee that three cubes is enough for $\bigslant{\Bbb F_q[x]}{\langle p_r^{i_r} \rangle}$ for every $1 \leqslant r \leqslant j$, then Chinese remainder theorem guarantees that three cubes is enough for $\bigslant{\Bbb F_q[x]}{\langle p_1^{i_1} \rangle} \times \cdots \times \bigslant{\Bbb F_q[x]}{\langle p_j^{i_j} \rangle}$. One last thing to note is that when we determine $m$, if $-1$ is a $k^{\text{\tiny th}}$ power (in particular if $k$ is odd) we can use the same $m$ from the downstairs $\left(\bigslant{\Bbb F_q[x]}{\langle p_r \rangle}\right)$; but if $-1$ is not a $k^{\text{\tiny th}}$ power (for example in some cases when $k$ is even) we need to increase $m$ by $1$ when we do the lifting using Proposition~\ref{fancy Henseling}. That is why in the table when $k$ is even, we increased $m$ by $1$ just to be safe without paying any attention to $-1$ being a $k^{\text{\tiny th}}$ power or not.
\end{proof}
Note if we have a finite semisimple ring, by Artin-Wedderburn theory it is a direct product of finitely many matrix rings over finite fields and we can use matrix rings results to get sharp bounds for the Waring problem on finite semisimple rings. For example if we know $R$ is semisimple with $|R|=5^{4}11^{2}13$, then we know every element of $R$ can be written as a sum of two cubes; since there is not any possibility that $R$ has a matrix ring over $\Bbb F_2$, $\Bbb F_4$ or a finite field of characteristic $3$ in its decomposition.

Next we give an example which shows that sometimes the matrix rings are better than the finite fields in terms of coverability, since we have more flexibility.

\textbf{Example.} Take $\Bbb F_4 \cong \bigslant{\Bbb F_2[x]}{\langle x^2+x+1 \rangle} $ so that we can denote the elements with $\{0,1,x,x+1\}$. We have $x^2=x+1$, $(x+1)^2=x^2 + 1 = x$ and $x(x+1)=1$. Let $A=                   
\begin{bmatrix}
		1 & 1 \\
		1 & x+1 \\
\end{bmatrix}$ and $B=
\begin{bmatrix}
x+1 & 1 \\
1 & 1 \\
\end{bmatrix}.$ \\
Then, $A^3=
\begin{bmatrix}
	x & 1 \\
	1 & 0 \\
	\end{bmatrix}$ and $B^3=
\begin{bmatrix}
0 & 1 \\
1 & x \\
\end{bmatrix}$. So, $A^3+B^3=
\begin{bmatrix}
	x & 0 \\
	0 & x \\
	\end{bmatrix}.$ This is very interesting although we cannot write $x$ as a sum of cubes in $\Bbb F_4$ (since $x^3= (x+1)^3 = 1$ in $\Bbb F_4$ ), we can write $ \begin{bmatrix}
	x & 0 \\
	0 & x \\
\end{bmatrix}$ as a sum of two cubes in $\operatorname{Mat}_2 (\Bbb F_4)$. Furthermore, one can also show that every element of $\operatorname{Mat}_2 (\Bbb F_4)$ can be written as a sum of two cubes. 

\section{Waring's Problem in Finite Rings}\label{WP in Finite Rings}

Let $R$ be a finite ring with identity which is not necessarily commutative, and let $k$ be any positive integer. In this section we want to find an integer $n$ with the property that every element of $R$ can be written as a sum of at most $n$ many $k^{\text{\tiny th}}$ powers in $R$. Before we dive into results and proofs, first we need to setup our machinery for finite rings with identity. First notice $\bigslant{R}{J}$ is semisimple i.e. $J \Big( \bigslant{R}{J} \Big) = 0$. Moreover, since $R$ is finite, $\bigslant{R}{J}$ is finite so $\bigslant{R}{J}$ is both left and right Artinian. Theorem~\ref{AWT} implies that $\bigslant{R}{J} \cong \operatorname{Mat}_{n_{1}}(D_1) \times \cdots \times \operatorname{Mat}_{n_{r}}(D_r)$ for some $D_1, \cdots D_r$ division rings. Since $\bigslant{R}{J}$ is finite, each $D_i$ has to have finitely many elements. By Wedderburn's little theorem $D_i$'s are finite fields. Therefore, we have $\bigslant{R}{J} \cong \operatorname{Mat}_{n_{1}}(\Bbb F_{q_1}) \times \cdots \times \operatorname{Mat}_{n_{r}}(\Bbb F_{q_r})$ for some finite fields $\Bbb F_{q_1}, \cdots, \Bbb F_{q_r}$. Furthermore \[ J\supseteq J^2 \supseteq J^3 \supseteq \cdots \supseteq J^l = (0)\] for some $l$ by Corollary~\ref{Nak2} and \[ R= \bigslant{R}{J^l} \twoheadrightarrow \bigslant{R}{J^{l-1}} \twoheadrightarrow \cdots \bigslant{R}{J^2} \twoheadrightarrow \bigslant{R}{J} \cong \operatorname{Mat}_{n_{1}}(\Bbb F_{q_1}) \times \cdots \times \operatorname{Mat}_{n_{r}}(\Bbb F_{q_r}).\]

\begin{thm} \label{Waring-ring}
Let $R$ be a finite ring with identity which is not necessarily commutative and $J$ denote the Jacobson radical of $R$. Let $k$ be a positive integer such that $\gcd(|R|,k)=1$ where $|R|$ denotes the order of the ring.
\vspace{-0.3cm} 
\begin{itemize}
\item If every element of $\bigslant{R}{J}$ is a $k^{\text{\tiny th}}$ power, then unit elements of $R$ are $k^{\text{\tiny th}}$ powers and nonunit elements of $R$ can be written as a sum of two $k^{\text{\tiny th}}$ powers. 
\vspace{-0.3cm} 
\item Let $m>1$. Assume every element of $\bigslant{R}{J}$ can be written as a sum of $m$ many $k^{\text{\tiny th}}$ powers in $\bigslant{R}{J}$. If $\alpha \in R$ is a unit  (or equivalently if $\alpha$ $\bmod \ J$ in $\bigslant{R}{J}$ is a unit), then $\alpha$ can be written as a sum of $m$ many $k^{\text{\tiny th}}$ powers in $R$. If $k$ is odd (resp. even), then every element of $R$ can be written as a sum of $m$ (resp. $m+1$) many $k^{\text{\tiny th}}$ powers in $R$. 
\end{itemize}
\end{thm}
\vspace{-0.5cm} 
\begin{proof}
\textbf{Case 1:} Let $R$ be commutative. Then we have $\bigslant{R}{J} \cong \Bbb F_{q_1} \times \cdots \times \Bbb F_{q_r}$ for some finite fields, $\Bbb F_{q_i}$'s. Let $\alpha \in R$. Since we assumed every element of  $\bigslant{R}{J}$ can be written as a sum of $m$ many $k^{\text{\tiny th}}$ powers in $\bigslant{R}{J}$, we know $\alpha \equiv B_1^k + \cdots+ B_m^k$ $\bmod \ J$ for some $B_1, \cdots, B_m \in R$. We can send $\alpha$ from $R$ to $\bigslant{R}{J}$ with the reduction $\bmod \ J$, and denote it with $\bar{\alpha}$ so that $\bar{\alpha}=(\alpha_1,\alpha_2,\cdots, \alpha_r)$ for some $\alpha_i \in \Bbb F_{q_i}$ where $1 \leqslant i \leqslant r$. Similarly, we will use $\bar{B}_j$ to denote $B_j$ $\bmod \ J$ for every $B_j \in R$.
 
\textbf{Case 1.a:} If $\bar{\alpha}$ is a unit in $\bigslant{R}{J}$ that means each entry of $\bar{\alpha}$ (i.e. $\alpha_i$ for every $1 \leqslant i \leqslant r$) is nonzero. Then we can arrange $\bar{B}_j$'s (by shuffling the entries of $\bar{B}_j$'s) such that $\bar{B}_1$ does not have any zero entry, so that $\bar{B}_1$ is a unit in $\Bbb F_{q_1} \times \cdots \times \Bbb F_{q_r}$. Since we have $\gcd(|R|,k)=1$ by assumption, $\bar{k}=(k_1, k_2, \cdots, k_r)$ that is the image of $k$ under reduction $\bmod \ J$ is a unit in $\Bbb F_{q_1} \times \cdots \times \Bbb F_{q_r}$. This implies $\bar{k} \bar{B}_1^{k-1}$ is a unit, and the result follows from Corollary~\ref{Hensel2}.

\textbf{Case 1.b:}
If $\bar{\alpha}$ is not a unit in $\bigslant{R}{J}$, that means at least one of the entries of $\bar{\alpha}$ is zero. Assume only one of them equals $0$, say $\alpha_1$. We have $\alpha \equiv B_1^k + \cdots+ B_m^k$ $\bmod \ J$ by assumption. We can shuffle the entries of $\bar{B}_j$'s such that the $i^{\text{\tiny th}}$ entry of $\bar{B}_1$ is nonzero for every $2 \leqslant i \leqslant r$. If the first entry of one of the $\bar{B}_j$'s is nonzero, then we can shuffle the nonzero entry to $\bar{B}_1$ and $\bar{B}_1$ becomes again a unit, we are back in case 1.a. If the first entries of all of the $\bar{B}_j$'s are zero for $1 \leqslant j \leqslant m$, then we can do the following trick. We can replace $0$ in the first entry of $\bar{B}_1$ with $1$, so that $\bar{B}_1$ becomes a unit. But then the first entry of $\bar{B}_1^k$ also becomes $1$ and we need to destroy the effect of $1$ in the summation ($\bar{B}_{1}^k+\bar{B}_{2}^k+ \cdots+ \bar{B}_{m}^k=\bar{\alpha}$). Therefore, we have to change the first entries of $\bar{B}_j$'s for $2 \leqslant j \leqslant m$ such that the first entry of the summation $\bar{B}_2^k + \cdots+ \bar{B}_m^k$ will be $-1$. If $k$ is odd, then simply replace $0$ in the first entry of $\bar{B}_2$ with $-1$, leave the other $\bar{B}_j$'s as they are and the problem is solved since $(-1)^k=-1$. If $k$ is even, by assumption $(-1,0,\cdots,0)$ can be written as a sum of $m$ many $k^{\text{\tiny th}}$ powers in $\bigslant{R}{J}$, this implies for some $x_j \in \Bbb F_{q_1}$ we have $(-1,0,\cdots,0)= (x_1,0,\cdots,0)^k+(x_2,0,\cdots,0)^k+\cdots+(x_m,0,\cdots,0)^k$. Let $\bar{B}_{new}=(x_1,0,\cdots,0)$, and replace the first entries of $\bar{B}_j$'s with $x_j$'s for $2 \leqslant j \leqslant r$, so that the first entry of the summation $\bar{B}_{new}^k+\bar{B}_2^k + \cdots+ \bar{B}_m^k$ will be $-1$ and $\bar{\alpha}$ will be equal to $\bar{B}_{1}^k+ \bar{B}_{new}^k +\bar{B}_{2}^k+ \cdots+ \bar{B}_{m}^k$. We again have $\bar{k} \bar{B}_1^{k-1}$ is a unit, since $\gcd(|R|,k)=1$ and $\bar{B}_1$ is a unit. Corollary~\ref{Hensel2} implies that $\alpha$ can be written as a sum of $m+1$ many $k^{\text{\tiny th}}$ powers.

\textbf{Case 2:} Let $R$ be noncommutative. To reduce this case to the commutative case, we use the following trick. Let $\alpha \in R$. Consider the subring generated by $\alpha$, i.e. $\mathbb{Z}[\alpha]= \left\{ a_{0}.1+a_{1}\alpha+\cdots+a_{l}\alpha^{l} \mid a_0, a_1 \cdots a_l \in \mathbb{Z} \right\}.$ It is a subring but it is a commutative finite ring with identity by itself. Moreover, since the order of $\mathbb{Z}[\alpha]$ has to divide $|R|$, $\gcd(|R|,k)=1$ implies $\gcd(|\mathbb{Z}[\alpha]|,k)=1$ and we are again in case $1$. Note that this approach enables us to write every element of $\mathbb{Z}[\alpha]$ as a sum of $k^{\text{\tiny th}}$ powers from $\mathbb{Z}[\alpha]$ instead of $R$, which means when $R$ is not commutative the $n$ value (i.e. the number of $k^{\text{\tiny th}}$ powers needed to write $\alpha$ as a sum) we find with this approach can be bigger than the actual $n$ value. 
\end{proof}
We actually proved a stronger result in the last proof. It follows that in the second part of Theorem~\ref{Waring-ring} if $-1$ is a $k^{\text{\tiny th}}$ power in $R$, then every element of $R$ can be written as a sum of $m$ many $k^{\text{\tiny th}}$ powers in $R$, otherwise we need to increase $m$ by $1$.
\begin{thm} \label{importantTable}
Let $R$ be a finite ring with identity which is not necessarily commutative. In the following table in a fixed row, let $k$ be a positive integer as in the first column. If none of the $q$ values stated in the second column divides the order of the ring, then every element of $R$ can be written as a sum of $n$ many $k^{\text{\tiny th}}$ powers in $R$. In fact, in this case any $\alpha \in R$ has $\alpha=B_{1}^k+B_{2}^k+\cdots+B_{n}^k$ with $B_{1}$, $B_{2}, \cdots, B_{n}$ in $\mathbb{Z}[\alpha]$. (For example the first row should be read as "If $3,4 \nmid |R|$, then any element of $R$ can be written as a sum of three cubes in $R$. In fact, in this case any $\alpha \in R$ has $\alpha=B_{1}^3+B_{2}^3+B_{3}^3$ with $B_{1}$, $B_{2}$, $B_{3}$ in $\mathbb{Z}[\alpha]$.")
\end{thm}

\begin{longtable}{ | p{1.8cm} | p{12cm} | p{1.8cm} | }
\hline
\boldmath$k=$ &  \boldmath$q \nmid $ & \boldmath$n=$\\
\hline \hline
\multirow{2}{*}{} $3$ & $3$, $4$ & $3$ \\ \cline{2-2} \cline{3-3} 
                  	         & $3$, $4$, $7$ & $2$ \\ \cline{2-2} \cline{3-3} \hline  \hline
\multirow{3}{*}{} $4$ & $2$, $9$ & $5$ \\ \cline{2-2} \cline{3-3} 
                  	         & $2$, $5$, $9$ & $4$ \\ \cline{2-2} \cline{3-3} 
                  	         & $2$, $5$, $9$, $13$, $17$, $29$ & $3$ \\ \cline{2-2}  \cline{3-3} \hline  \hline
\multirow{3}{*}{} $5$ & $5$, $16$ & $5$ \\ \cline{2-2} \cline{3-3} 
                  	         & $5$, $11$, $16$ & $3$ \\ \cline{2-2} \cline{3-3} 
                  	         & $5$, $11$, $16$, $31$, $41$, $61$ & $2$ \\ \cline{2-2}  \cline{3-3} \hline  \hline	      
\multirow{4}{*}{} $6$ & $2$, $3$, $25$ & $7$ \\ \cline{2-2} \cline{3-3} 
					    & $2$, $3$, $7$, $13$, $25$ & $5$ \\ \cline{2-2} \cline{3-3} 
                  	         & $2$, $3$, $7$, $13$, $19$, $25$, $31$ & $4$ \\ \cline{2-2} \cline{3-3} 
                  	         & $2$, $3$, $7$, $13$, $19$, $25$, $31$, $37$, $43$, $61$, $67$, $73$, $79$, $109$, $139$, $223$  & $3$ \\  \cline{2-2}  \cline{3-3} \hline  \hline
\multirow{3}{*}{} $7$ &  $7$, $8$ & $4$  \\ \cline{2-2} \cline{3-3} 
                  	         &  $7$, $8$, $29$, $43$ & $3$ \\ \cline{2-2} \cline{3-3} 
                  	         &  $7$, $8$, $29$, $43$, $71$, $113$, $127$ & $2$ \\ \cline{2-2}  \cline{3-3} \hline  \hline
\multirow{4}{*}{} $8$ &  $2$, $9$, $49$ & $9$ \\ \cline{2-2} \cline{3-3} 
                  	       &  $2$, $9$, $17$, $49$ & $5$ \\ \cline{2-2} \cline{3-3} 
	      			  &  $2$, $5$, $9$, $17$, $41$, $49$ & $4$ \\ \cline{2-2}  \cline{3-3}
                  	      &  $2$, $5$, $9$, $13$, $17$, $29$, $41$, $49$, $73$, $89$, $97$, $113$, $121$, $137$, $233$, $257$, $337$, $761$ & $3$ \\ \cline{2-2}  \cline{3-3} \hline  \hline
\multirow{4}{*}{} $9$ & $3$, $4$ & $9$  \\ \cline{2-2} \cline{3-3} 
                  	       & $3$, $4$, $19$ & $5$ \\ \cline{2-2} \cline{3-3} 
	      			  & $3$, $4$, $19$, $37$ & $3$ \\  \cline{2-2} \cline{3-3} 
                  	       & $3$, $4$, $7$, $19$, $37$, $73$, $109$, $127$, $163$, $181$, $199$, $271$, $307$ & $2$ \\ \cline{2-2}  \cline{3-3} \hline  \hline	      
\multirow{5}{*}{} $10$ & $2$, $5$, $81$ & $11$ \\ \cline{2-2} \cline{3-3} 
                  	           & $2$, $5$, $11$, $81$ & $6$ \\ \cline{2-2} \cline{3-3} 
	      			      & $2$, $5$, $11$, $31$, $81$ & $5$ \\ \cline{2-2} \cline{3-3} 
				 	      & $2$, $5$, $11$, $31$, $41$, $61$, $81$ & $4$ \\ \cline{2-2} \cline{3-3} 
                  	           & $2$, $5$, $11$, $31$, $41$, $61$, $71$, $81$, $101$, $131$, $151$, $181$, $191$, $211$, $241$, $251$, $271$, $281$, $311$, $331$, $401$, $421$, $431$, $461$, $491$, $641$, $911$ & $3$ \\ \cline{2-2}  \cline{3-3} \hline  \hline    
\multirow{4}{*}{} $11$ & $11$, $23$ & $5$ \\ \cline{2-2} \cline{3-3} 
                  	           & $11$, $23$, $67$ & $4$ \\ \cline{2-2} \cline{3-3} 
	      			      & $11$, $23$, $67$, $89$ & $3$ \\ \cline{2-2} \cline{3-3} 
				 	      & $11$, $23$, $67$, $89$, $199$, $331$, $353$, $419$, $463$, $617$ & $2$ \\ \cline{2-2} \cline{3-3} \hline     	              
\end{longtable}

\begin{proof}
Given a $k$ value, we want to find an integer $n$ with the property that every element of $R$ can be written as a sum of at most $n$ many $k^{\text{\tiny th}}$ powers in $R$. We want to use the previous theorem. One of the assumptions in that theorem is $\gcd(|R|,k)=1$. In the table, notice that $q$-column always has the divisors of $k$. So when we say $q$ does not divide $|R|$ in the statement, it is guaranteed that $\gcd(|R|,k)=1$.

If $R$ is commutative, then $\bigslant{R}{J}$ is isomorphic to the direct product of some finite fields as we discussed earlier. If $R$ is not commutative, then let $\alpha \in R$. Since $\mathbb{Z}[\alpha]$ is a commutative finite ring with identity, we have $\bigslant{\mathbb{Z}[\alpha]}{J \left( \mathbb{Z}[\alpha] \right)} \cong \Bbb F_{q_1} \times \cdots \times \Bbb F_{q_r}$ for some finite fields $\Bbb F_{q_1}, \cdots, \Bbb F_{q_r}$. That means to find $m$ in the previous theorem, first we need to find how many $k^{\text{\tiny th}}$ powers is necessary to write every element of $\Bbb F_{q_i}$ as a sum of $k^{\text{\tiny th}}$ powers in $\Bbb F_{q_i}$. We denote this number with $m_i$ for each $\Bbb F_{q_i}$, i.e. every element of $\Bbb F_{q_i}$ can be written as a sum of $m_i$ many $k^{\text{\tiny th}}$ powers in $\Bbb F_{q_i}$. Then, every element of $\bigslant{R}{J}$ $\left( \text{or resp. } \bigslant{\mathbb{Z}[\alpha]}{J \left( \mathbb{Z}[\alpha] \right)}\right)$ can be written as a sum of $m=\max_{1 \leqslant i \leqslant r} m_i$ many $k^{\text{\tiny th}}$ powers in $\bigslant{R}{J}$ $\left( \text{or resp. in } \bigslant{\mathbb{Z}[\alpha]}{J \left( \mathbb{Z}[\alpha] \right)}\right)$. Therefore, building on our results in Section~\ref{fields section} we can determine $m$, and using the previous theorem we can let $n=m$ or $n=m+1$ depending on the parity of $k$.
\end{proof}

Notice that in the previous theorem, we did not note the rings in which every element is a $k^{\text{\tiny th}}$ power, i.e. there are not any rows with $n=1$ in the table. The following propositions are stated with this purpose, and their proofs follow from Proposition~\ref{gamma1 case} and Theorem~\ref{Waring-ring} easily. 

\begin{prop}
Let $R$ be a finite ring with identity and with a cube-free order. Let $|R|=p_1^{i_{1}}p_2^{i_{2}} \cdots p_s^{i_{s}}$ be the prime factorization of the order of the ring. Let $k$ be a positive integer such that $p_j \nmid k$ for any $j$, and $\gcd(k,p_j^{i_j}-1)=1$ for all $1 \leqslant j \leqslant s$. If $\alpha \in R$ is a unit, then $\alpha$ is a $k^{\text{\tiny th}}$ power in $R$. If $k$ is odd, then any element of the ring is a $k^{\text{\tiny th}}$ power. If $k$ is even, then any element of the ring can be written as a sum of two $k^{\text{\tiny th}}$ powers.
\end{prop}

\begin{prop}
Let $R$ be a finite ring with identity. Let $|R|=p_1^{i_{1}}p_2^{i_{2}} \cdots p_s^{i_{s}}$ be the prime factorization of the order of the ring. Let $k$ be a positive integer such that $p_j \nmid k$ for any $j$, and $\gcd \left(k, \prod_{\iota=1}^{i_j}(p_j^{\iota}-1) \right)=1$ for all $1 \leqslant j \leqslant s$. If $\alpha \in R$ is a unit, then $\alpha$ is a $k^{\text{\tiny th}}$ power in $R$. If $k$ is odd, then any element of the ring is a $k^{\text{\tiny th}}$ power. If $k$ is even, then any element of the ring can be written as a sum of two $k^{\text{\tiny th}}$ powers.
\end{prop}

For example consider the group algebra $\Bbb F_3[T]$ where $T$ is the group of upper triangular $3 \times 3$ matrices with $1$'s in the diagonal over $\Bbb F_3$ i.e. $ T= \left\{  \begin{bmatrix}
	1 & a & b\\
	0 & 1 & c\\
	0 & 0 & 1\\
	\end{bmatrix} \  \bigg | \ a,b,c \in \Bbb F_3 \right\}$. $T$ is a group under multiplication and nonabelian, this implies $\Bbb F_3[T]$ is noncommutative. Moreover, Maschke's theorem does not apply, the group algebra is not semisimple. But via the last proposition we can conclude that every element of $\Bbb F_3[T]$ is a septic ($7^{\text{\tiny th}}$ power), since $3 \nmid 7$ and $\gcd(7, 2 \times 8 \times 26 \times 80)=1$.
	
\appendix
\section{Lemmata} \label{Lemmata}
Here we prove the basic inequalities used in Section~\ref{fields section}.
\begin{lem} \label{induction1}
Let $x,y \in \mathbb{Z}_{+} \setminus \{1\}$. If $y>x^4$, then $(y-1)^4-x^4y^3+(y-1)y^2x^3>0.$
\end{lem}

\begin{proof}
Step $1$: Pick any $x \in \mathbb{Z}_{+} \setminus \{1\}$ and fix it.\\
We will prove $F(x,y)=(y-1)^4-x^4y^3+(y-1)y^2x^3>0$ when $y=x^4+1$. 
\begin{align*}
F(x,x^4+1) &=x^{16}-x^4(x^4+1)^3+x^4(x^4+1)^2x^3\\
&=x^4(x^3-x-1)(x^8+x^6-2x^5+3x^4-x^3+x^2-x+1).
\end{align*}
$x^4 \geqslant 0$ for every $x \in \mathbb{R}$.\\
$x^3-x-1$ has only one real root in the interval $(1,2)$ and two complex conjugate roots, as the first derivative test and sign chart illustrates it. We have $x^3-x-1>0$ when $x \in \mathbb{Z}_{+} \setminus \{1\}$.\\
Also, $x^8+x^6-2x^5+3x^4-x^3+x^2-x+1>0$ as $x^8+x^6> 2x^5$ and $2x^4>x^3+x$ when $x \in \mathbb{Z}_{+} \setminus \{1\}$.\\
As a result, we proved $F(x,y)>0$ when $y=x^4+1$ and $x \in \mathbb{Z}_{+} \setminus \{1\}$.\\
Step $2$: Pick any $x \in \mathbb{Z}_{+} \setminus \{1\}$. Now, we will prove by induction that $F(x,y)>0$ for any $y>x^4$. We already showed in Step $1$ that $F(x,y)>0$ when $y=x^4+1$. We assume this holds for $y=x^4+C$ for any $C \in \mathbb{Z}_{+}$, and we demonstrate below $F(x,y)>0$ also holds when $y=x^4+C+1.$\\
We have \[F(x,x^4+C)=(x^4+C-1)^4-x^{4}(x^4+C)^3+(x^4+C-1)(x^4+C)^{2}x^{3}>0\] by assumption, and 
\[F(x,x^4+C+1)=(x^4+C+1-1)^4-x^4(x^4+C+1)^3+(x^4+C+1-1)(x^4+C+1)^{2}x^3.\]
Notice that if we show $B=[F(x,x^4+C+1)-F(x,x^4+C)]>0$, then we are done.\\
We have
\[B= 4 C^{3} + 9 C^{2} x^{4} + \underbrace{3 C^2 x^3 - 6 C^2} + 6 C x^8 + \underbrace{6 C x^7 - 15 C x^4} + C x^3 + 4 C + x^{12} + \underbrace{3 x^{11} - 9 x^8} + x^7 + \underbrace{3 x^4 - 1}.\]
Notice that we have 
\[3 C^2 x^3 - 6 C^2=3C^2(x^3-2)>0\] 
\[6 C x^7 - 15 C x^4=3Cx^4(2x^3-5)>0\] 
\[3 x^{11} -9 x^8=3x^8(x^3-3)>0 \text{ \ \ \ and \ \ \ } 3 x^4 - 1>0\] since $x \in \mathbb{Z}_{+} \setminus \{1\}$. These calculations show $B>0$, and the result follows.
\end{proof}

\begin{lem} \label{induction2}
Let $x,y \in \mathbb{Z}_{+} \setminus \{1\}$. If $y>x^3$, then $(y-1)^3-x^{2}y(xy-y+1)>0.$
\end{lem}

\begin{proof}
We again use induction to prove the claim.\\
Step $1$: Pick any $x \in \mathbb{Z}_{+} \setminus \{1\}$ and fix it.\\
We need to show $F(x,y)=(y-1)^3-x^{2}y(xy-y+1)>0$ holds when $y=x^3+1$.\\
We have $F(x,x^3+1)=x^8-2x^6+x^5-x^3$ is clearly bigger than zero as $x \in \mathbb{Z}_{+} \setminus \{1\}$.\\
Step $2$: Pick any $x \in \mathbb{Z}_{+} \setminus \{1\}$. We assume the claim holds for $y=x^3+C$ for any $C \in \mathbb{Z}_{+}$, and we demonstrate below $F(x,y)>0$ holds also when $y=x^3+C+1.$\\
We have \[F(x,x^3+C)=(x^{3}+C-1)^3-x^2(x^3+C) [x(x^3+C)-(x^3+C)+1]>0\] by assumption, and 
\[F(x,x^3+C+1)=(x^{3}+C+1-1)^3-x^{2}(x^3+C+1)[x(x^{3}+C+1)-(x^{3}+C+1)+1].\]
Notice that if we show $B=[F(x,x^3+C+1)-F(x,x^3+C)]>0$, then we are done.\\
We have
\[B= 3 C^{2} + 4 C x^{3} + \underbrace{2 C x^2 - 3 C} + x^6 + \underbrace{2 x^5 - 4 x^3}+1.\]
Notice that both $2 C x^2 - 3 C>0$ and $2 x^5 - 4 x^3>0$ as $x \in \mathbb{Z}_{+} \setminus \{1\}$. So, we have $B>0$ and the result follows.
\end{proof}

\section{Computer Code}
The computer search was carried out on a supercomputer using a computer program called Sage (ours was Sage/6.9). The code used to generate Table~\ref{table1},~\ref{table2},~\ref{table3} and the data in Theorem~\ref{gamma bigger than or equal to 7-thm} is provided below. Inside the code there are some parts written after the pound sign $\#$. These are just comments for better understanding the code, and these comments are ignored by the computer. Also, note that we provided the original, indented code below. If you want to use the same code for some calculation, you need to preserve this indentation.

k=37
\begin{flalign*}
\text{def } \ & f(n):    &  \# & \text{We define a new function } f(n) \text{ in the next four lines.} &\\
&A = [ \ ]        &  \# & A \text{ is an empty list.} &\\
&\text{for } j \text{ in FiniteField}(n, '\mkern-\thinmuskip a'):  & \# & \text{FiniteField}(n, '\mkern-\thinmuskip a') \text{ is the finite field with order } n. &\\
&\qquad A.\text{append}(j**k) & \# & \text{This command adds } j^k \text{ to } A. &\\
&\text{return uniq}(A) & \# & \text{uniq}(A) \text{ removes the same entries of the array } A. &\\
\end{flalign*}
\begin{flalign*}
\text{def } \ &f2(n): &\\
	 		     &A2 = [ \ ] &\\
			     &\text{for } x \text{ in } f(n): &\\
	 		     &\qquad \text{for } y \text{ in } f(n): &\\
			     &\qquad \qquad  A2.\text{append}(x+y) &\\
			     &\text{return uniq}(A2) & \# & f2(n) \text{ gives the list of the elements in } 2R_{k}. &\\
\end{flalign*}
\begin{flalign*}
\cdots 		& \ \  & \# & \text{Similarly, we have the definition of } fX \text{ here for every } X \text{ between } 2 \text{ and }41. &\\
\end{flalign*}
\begin{flalign*}
\text{def } \ &f41(n): &\\
	 		     &A41 = [ \ ] &\\
			     &\text{for } x \text{ in } f40(n): &\\
	 		     &\qquad \text{for } y \text{ in } f(n): &\\
			     &\qquad \qquad A41.\text{append}(x+y) &\\
			     &\text{return uniq}(A41) & \# & f41(n) \text{ gives the list of the elements in } 41R_{k}. &\\
\end{flalign*}
\begin{flalign*}
			   &S = [ \ ] & \# & S \text{ will keep the list of the finite fields with order } n \text{ for } n \leqslant 8k^4. &\\
			   &\text{for } n \text{ in range} (1, 8k^4+1): &\\
	 		   &\qquad \text{if len} (\text{prime}\_\text{divisors}(n)) ==1: & \# & \text{len}(\text{prime}\_\text{divisors}(n)) \text{ gives the number of distinct prime divisors of } n. &\\
			   &\qquad \qquad S.\text{append}(n) &\\
\end{flalign*}
\begin{flalign*}
			   &S1 = [ \ ] & \# & S1 \text{ will keep the list of the finite fields with order } n \text{ such that } |R_k| \neq n. &\\
			   &\text{for } n \text{ in } S: &\\
	 		   &\qquad \text{if len}(f(n))!=n: &\\
			   &\qquad \qquad S1.\text{append}(n)&\\
\end{flalign*}
\begin{flalign*}
			   &S2 = [ \ ] & \# & S2 \text{ will keep the list of the finite fields with order } n \text{ such that } |2R_{k}| \neq n. &\\
			   &\text{for } n \text{ in } S1: &\\
	 		   &\qquad \text{if len}(f2(n))!=n: &\\
			   &\qquad \qquad S2.\text{append}(n)&\\
\end{flalign*}
\begin{flalign*}
\cdots 	   & \ \  & \# & \text{Similarly, we have the definition of } SX \text{ here for every } X \text{ between } 2 \text{ and } 41.&\\\end{flalign*}
\begin{flalign*}
			   &S41 = [ \ ] & \# & S41 \text{ will keep the list of the finite fields with order } n \text{ such that } |41R_{k}| \neq n. &\\
			   &\text{for } n \text{ in } S40: &\\
	 		   &\qquad \text{if len}(f41(n))!=n: &\\
			   &\qquad \qquad S41.\text{append}(n)&\\ 
\end{flalign*}
\begin{flalign*}
& \text{print }  '\mkern-\thinmuskip k=', \text{k} &\\
& \text{print }  '\mkern-\thinmuskip S1',':', \text{S1} &\\
& \text{print }  '\mkern-\thinmuskip S2',':', \text{S2} &\\
& \cdots \qquad \qquad \qquad \# \text{Similarly, we have the same command for } SX \text{ here for every } X \text{ between } 2 \text{ and } 41. &\\
& \text{print } '\mkern-\thinmuskip S41',':', \text{S41} &\\
\end{flalign*}

\section*{Acknowledgement}
I would like to express my sincere gratitude to my advisers, Professor Jonathan Pakianathan and Professor David Covert for suggesting this problem. I also thank University of Rochester and the CIRC team for letting me use their supercomputer. This work was partially supported by NSA grant H98230-15-1-0319.

%\bibliography{trial}
%\bibliographystyle{plain}

\end{document}